\documentclass[10pt]{amsart}

\usepackage{times,amsfonts,amsmath,amstext,amsbsy,amssymb,
  amsopn,amsthm,upref,eucal, amscd}
\usepackage[T1]{fontenc}

\usepackage[pdftex]{graphicx}

\newtheorem{theorem}{Theorem}[section]
\newtheorem{lemma}[theorem]{Lemma}

\newtheorem{proposition}[theorem]{Proposition}

\numberwithin{equation}{section}

\theoremstyle{definition}
\newtheorem{definition}[theorem]{Definition}

\newtheorem{remark}[theorem]{Remark}

\newcommand{\wt}{\widetilde}
\newcommand{\wh}{\widehat}
\newcommand{\Zq}{\mathbb Z /q}
\newcommand{\bs}{\breve{\sigma}}
\newcommand{\bz}{\breve{\zeta}}
\begin{document}

\title[The affine group of certain exceptionally symmetric origamis]{The action of the affine diffeomorphisms on the relative homology group of certain exceptionally symmetric origamis}

\author{Carlos Matheus and Jean-Christophe Yoccoz}
\address{Coll\`ege de France, 3 Rue d'Ulm, 75005, Paris, France}
\email{matheus@impa.br, jean-c.yoccoz@college-de-france.fr.}

\date{\today}

\begin{abstract}
We compute explicitly the action of the group of affine diffeomorphisms on the relative homology of two remarkable origamis
 discovered respectively by Forni (in genus 3) and Forni-Matheus (in genus 4). We show that, in both cases, the action on the non trivial  part of the
  homology is through finite groups. In particular, the action on some $4$-dimensional invariant subspace of the homology leaves invariant a root system of $D_4$ type.
  This provides as a by-product a new 
 proof of (slightly stronger versions of) the results of Forni and Matheus: the non trivial Lyapunov exponents of the Kontsevich-Zorich cocycle for the Teichm\"uller disks of these two origamis are equal to zero.
\end{abstract}
\maketitle

\tableofcontents


\section{Introduction}\label{intro}

Our main objective is the explicit description of the action on homology of the affine group of the square-tiled translation surfaces constructed by  Forni~\cite{ForniSurvey} and Forni- Matheus \cite{FM}, characterized by the total degeneracy of the so-called \emph{Kontsevich-Zorich cocycle}. Before going to the statements of 
our results, let us briefly recall some basic material about these notions.

\subsection{The Teichm\"uller flow and the Kontsevich-Zorich cocycle}

Let $M$ be a closed oriented surface  of genus $g\geq 1$, and $\Sigma$ be a finite subset of $M$. We denote by $\textrm{Diff}^+(M,\Sigma)$ the group of
orientation-preserving 
 homeomorphisms of $M$ which preserve $\Sigma$, by $\textrm{Diff}_0^+(M,\Sigma)$ the connected component of the identity in $\textrm{Diff}^+(M,\Sigma)$ 
(i.e., $\textrm{Diff}_0^+(M,\Sigma)$ is the subset of homeomorphisms in $\textrm{Diff}^+(M,\Sigma)$ which are isotopic to the identity rel. $\Sigma$) 
and by
$\Gamma(M,\Sigma):=\textrm{Diff}^+(M,\Sigma)/\textrm{Diff}_0^+(M,\Sigma)$ 
the \emph{mapping class group} of $(M,\Sigma)$. When $\Sigma$ is empty, we just write $\textrm{Diff}^+(M)$, $\textrm{Diff}_0^+(M)$, $\Gamma(M)$ for $\textrm{Diff}^+(M,\Sigma)$, $\textrm{Diff}_0^+(M,\Sigma)$, $\Gamma(M,\Sigma)$.
\begin{definition} A {\it translation surface} structure on $M$ is a complex structure on $M$ together with a non-zero holomorphic 1-form $\omega$ w.r.t. this complex structure. The \emph{Teichm\"uller space} $\mathcal Q (M)$ (resp. the \emph{moduli space} 
$ \mathcal{M}(M)$) is the space of orbits for the natural action of $\textrm{Diff}_0^+(M)$ (resp. of $\textrm{Diff}^+(M)$) on the space of translation surface structures. We have thus $ \mathcal{M}(M)= \mathcal Q (M) /\Gamma(M)$.
\end{definition}

The group $SL(2,\mathbb{R})$ acts naturally on $\mathcal{M}(M)$ by postcomposition on the charts defined by local primitives of the
 holomorphic $1$-form. The \emph{Teichm\"uller  flow} $G_t$ is the restriction of the action to the diagonal subgroup
 $\textrm{diag}(e^{t},e^{-t})$ of $SL(2,\mathbb{R})$ on $\mathcal{M} (M)$. For later reference, we recall some of the main structures of
 the Teichm\"uller space $\mathcal Q (M)$ and the moduli space $\mathcal{M} (M)$:

\begin{itemize}
\item $\mathcal{M} (M)$ is stratified into analytic spaces $\mathcal{M} (M,\kappa)=\mathcal{M}_{\kappa}$  obtained by
 fixing the multiplicities $\kappa=(k_1,\dots,k_\sigma)$ of the set of zeros $\Sigma = \{p_1,\dots,p_\sigma\}$ of the holomorphic $1$-form (here $\sum k_i=2g-2$); the corresponding Teichm\"uller space $\mathcal Q (M,\Sigma,\kappa)$ is the space of orbits for the natural action of $\textrm{Diff}_0^+(M,\Sigma)$ on the set of translation surface structures with prescribed zeroes in $\Sigma$. One has $ \mathcal{M}(M,\kappa)= \mathcal Q (M,\Sigma,\kappa) /\Gamma(M,\Sigma)$;
\item The total area function $A:\mathcal{M} (M) \rightarrow \mathbb{R}^+$, $A(\omega) = \frac{i}{2} \int_M \omega \wedge \overline \omega$ is
 $SL(2,\mathbb{R})$-invariant so that the unit bundle $\mathcal{M}^{(1)}(M):=A^{-1}(1)$ and its strata $\mathcal{M}^{(1)}(M,\kappa):=
 \mathcal{M} (M,\kappa)\cap \mathcal{M}^{(1)}(M)$ are $SL(2,\mathbb{R})$-invariant (and, \emph{a fortiori}, $G_t$-invariant);
\item the Teichm\"uller space $\mathcal Q (M,\Sigma,\kappa)$  has a locally  affine structure modeled on the complex vector space $H^1(M,\Sigma,\mathbb{C})$: 
the local charts are given by the \emph{period map} defined by integrating the holomorphic $1$-form
 against the homology classes in $H_1(M,\Sigma,\mathbb{Z})$;
\item the Lebesgue measure on the Euclidean space $H^1(M,\Sigma,\mathbb{C})$ induces an absolutely continuous $SL(2,\mathbb{R})
$-invariant measure $\mu_\kappa$ on $\mathcal{M} (M,\kappa)$ such that the conditional measure $\mu_\kappa^{(1)}$ induced on
 $\mathcal{M}^{(1)}(M,\kappa)$ is invariant by the $SL(2,\mathbb{R})$-action (and hence $G_t$-invariant).
\end{itemize}
See Veech~\cite{Veech},~\cite{Veech2}, and the surveys of Yoccoz~\cite{Y-Pisa} and Zorich~\cite{Zorich3} for more details.

Once we get the existence of a natural invariant measure $\mu_\kappa^{(1)}$ for the Teichm\"uller flow, it is natural to ask whether
 $\mu_\kappa^{(1)}$ has finite mass and/or $\mu_\kappa^{(1)}$ is ergodic with respect to the Teichm\"uller dynamics. In this direction,
 we have the following result:
\begin{theorem}[Masur~\cite{Masur1}, Veech~\cite{Veech}]The total volume of $\mu_\kappa^{(1)}$ is finite and the Teichm\"uller flow
 $G_t=\textrm{diag}(e^t,e^{-t})$ is ergodic on each connected component of $\mathcal{M} (M,\kappa)$ with respect to $\mu_\kappa^{(1)}$.
\end{theorem}

\begin{remark}Veech~\cite{Veech2} showed that the strata are \emph{not} always connected. More recently, Kontsevich and Zorich~\cite
{KZ}  gave a complete classification of the connected components of all strata of the moduli spaces of holomorphic $1$-forms 
(see Lanneau~\cite{Lanneau} for the same result for  quadratic differentials).
\end{remark}

 The \emph{Kontsevich-Zorich} cocycle $G_t^{KZ}$ is the quotient of the trivial
 cocycle $G_t\times id: \mathcal {Q}(M)\times H^1(M,\mathbb{R})\to \mathcal {Q}(M)\times H^1(M,\mathbb{R})$ by the action of the mapping class group 
$\Gamma (M)$. As 
the action is \emph{symplectic}, the Lyapunov exponents of $G_t^{KZ}$ with respect to any $G_t$-invariant ergodic probability $\mu$ are
 \emph{symmetric} w.r.t. $0$:
\begin{equation*}
1=\lambda_1^\mu\geq\dots\lambda_g^\mu\geq-\lambda_g^\mu\geq\dots\geq-\lambda_1^\mu=-1.
\end{equation*}
It turns out that the $g$ non-negative exponents $1=\lambda_1^\mu\geq \lambda_2^\mu\geq\dots\geq\lambda_g^\mu$  determine the Lyapunov
 spectrum of the 
Teichm\"uller flow\footnote{In fact, this is one of the motivation of the introduction of the Kontsevich-Zorich cocycle.}. Indeed, the
 Lyapunov exponents of the Teichm\"uller flow with respect to a $G_t$-invariant ergodic probability $\mu$ on
 $\mathcal{M} (M,\kappa)$ are
\begin{eqnarray*}
2=(1+\lambda_1^\mu)\geq (1+\lambda_2^\mu)\geq\dots\geq (1+\lambda_g^{\mu})\geq \overbrace{1=\dots=1}^{\sigma-1}\geq
(1-\lambda_g^\mu) \geq \\ \dots\geq(1-\lambda_2^\mu)\geq 0\geq-(1-\lambda_2^\mu)\geq\dots\geq-(1-\lambda_g^\mu)\geq \overbrace{-1=\dots=-1}^{\sigma-1} \\
\geq-(1+\lambda_g^\mu)\geq \dots\geq -(1+\lambda_2^\mu)\geq-(1+\lambda_1^\mu)=-2.
\end{eqnarray*}
See~\cite{Zorich1} and~\cite{K} for further details. On the other hand, Zorich and Kontsevich conjectured that the
 Lyapunov exponents of $G_t^{KZ}$ for the canonical absolutely continuous measure $\mu_\kappa^{(1)}$ are all non-zero (i.e.,
 non-uniform hyperbolicity) and distinct (i.e., all Lyapunov exponents have multiplicity 1). After the fundamental works of 
G. Forni~\cite{Forni} (showing the non-uniform hyperbolicity of $G_t^{KZ}$) and Avila-Viana~\cite{AV} (proving the simplicity of the
 Lyapunov spectrum), it follows that the Zorich-Kontsevich conjecture is true. In other words, the Lyapunov exponents of $G_t^{KZ}$ at a
 $\mu_\kappa^{(1)}$-generic point are all non-zero and they have multiplicity 1.

\subsection{The affine group of a translation surface}

Let $(M,\omega)$ be a translation surface, i.e let $\omega$ be a non-zero holomorphic 1-form w.r.t. some complex structure on $M$. We denote as above by $\Sigma$ the set of zeros of $\omega$.

\begin{definition} The {\it affine group} 
${\rm Aff}(M,\omega)$ of $(M,\omega)$ is the group of orientation preserving homeomorphisms of $M$ which preserve $\Sigma$ and are given 
by affine maps in the charts defined by local primitives of $\omega$. In these charts, the differential of an affine map is an element of $SL(2,\mathbb{R})$. We obtain in this way a homomorphism from ${\rm Aff}(M,\omega)$ into $SL(2,\mathbb{R})$. The {\it automorphism group} 
${\rm Aut}(M,\omega)$ of $(M,\omega)$ is the kernel of this homomorphism.
\end{definition}

\begin{definition} The image of this homomorphism is called the {\it Veech group} of $(M,\omega)$ and denoted by $SL(M,\omega)$.
It is a discrete subgroup of $SL(2,\mathbb{R})$, equal to the stabilizer of $(M,\omega)$ for the action of $SL(2,\mathbb{R})$ on $ \mathcal{M}(M)$.
 \end{definition}

One has thus an exact sequence
\begin{equation}
1 \longrightarrow {\rm Aut}(M,\omega) \longrightarrow {\rm Aff}(M,\omega) \longrightarrow SL(M,\omega) \longrightarrow  1.
\end{equation}

For a nice account on affine and Veech groups  see the survey of Hubert and Schmidt~\cite{HuSc}.

In genus $g\geq 2$, the affine group ${\rm Aff}(M,\omega)$ injects into $\Gamma(M)$ (and \emph{a fortiori} into $\Gamma(M,\Sigma)$): for elements with non trivial image in $SL(M,\omega)$, this can be viewed from the period map (see Veech~\cite{Veech3}); for elements in $ {\rm Aut}(M,\omega)$, this is a consequence of the Lefschetz fixed point theorem, as fixed points then have index 1.

 Consider the natural surjective map:
$$H^1(M,\Sigma,\mathbb{R})\to H^1(M,\mathbb{R})$$
between the relative and absolute cohomology groups of $(M,\Sigma)$. We denote by $H^1_{st}(M,\Sigma,\mathbb{R})$  the  subspace of $H^1(M,\Sigma,\mathbb{R})$ spanned by $dx=\Re(\omega)$ and $dy=\Im(\omega)$ and by $H^1_{st}$ its image in $H^1(M,\mathbb{R})$. The subspace
$H^1_{st}$ (and therefore also $H^1_{st}(M,\Sigma,\mathbb{R})$) is $2$-dimensional: this can be seen either as a standard fact from Hodge theory or more concretely from Veech's \emph{zippered rectangles construction} (see~\cite{Veech1},~\cite{Y-Pisa},~\cite{Zorich3}). We denote by $H^1_{(0)}$ the orthogonal of $H^1_{st}$ with respect to the exterior product in $ H^1(M,\mathbb{R})$. We have 
$$ H^1(M,\mathbb{R}) = H^1_{st} \oplus H^1_{(0)},$$
because the $2$-form $dx \wedge dy$ defines a
 non-zero element in $H^2(M,\mathbb {R})$.
\medskip

The intersection form defines a non-degenerate pairing between the homology groups $H_1(M-\Sigma,\mathbb{R})$ and $H_1(M,\Sigma,\mathbb{R})$, and also between $H_1(M,\mathbb{R})$ and itself.

Let $H_1^{(0)} \subset H_1(M,\mathbb{R})$ be the annihilator of $H_{st}^1$ and $H_1^{(0)}(M,\Sigma,\mathbb{R}) \subset H_1(M,\Sigma,\mathbb{R})$ be  the annihilator of $H^1_{st}(M,\Sigma,\mathbb{R})$. Both are codimension $2$ subspaces. We also introduce the orthogonal $H_1^{st}(M-\Sigma,\mathbb{R})\subset H_1(M-\Sigma,\mathbb{R})$ of $H_1^{(0)}(M,\Sigma,\mathbb{R})\subset H_1(M,\Sigma,\mathbb{R})$ with respect to the intersection form between $H_1(M-\Sigma,\mathbb{R})$ and $H_1(M,\Sigma,\mathbb{R})$,
and the orthogonal $H_1^{st}\subset H_1(M,\mathbb{R})$ of $H_1^{(0)}$ for the intersection form on $H_1(M,\mathbb{R})$. 

Observe that $H_1^{st}$
is the annihilator of $H^1_{(0)}$, and also the image of $H_1^{st}(M-\Sigma,\mathbb{R})$ under the map from $H_1(M-\Sigma,\mathbb{R})$ to $H_1(M,\mathbb{R})$. Both $H_1^{st}$ and 
$H_1^{st}(M-\Sigma,\mathbb{R})$ have dimension 2. 

 \medskip
 
 One has
 $$H_1(M,\mathbb{R})=H_1^{st}\oplus H_1^{(0)},$$
 $$H_1(M,\Sigma,\mathbb{R})=H_1^{st}\oplus H_1^{(0)}(M,\Sigma,\mathbb{R}).$$
 \medskip

 These decompositions only depend on the image of the translation surface structure in Teichm\"uller space; they are constant along $SL(2,\mathbb{R})$-orbits, invariant under the action of $\textrm{Aff}(M,\omega)$, and covariant under the action of the mapping class groups. The same is true for the decompositions of the cohomology groups. In particular, the decomposition
 $$ H^1(M,\mathbb{R}) = H^1_{st} \oplus H^1_{(0)}$$
 is invariant under the Kontsevich-Zorich cocycle $G_t^{KZ}$. The subbundle $H^1_{st}$ correspond to the extreme exponents $\pm 1$  of $G_t^{KZ}$. 

\begin{definition} The restriction of the Kontsevich-Zorich cocycle $G_t^{KZ}$ to the invariant subbundle $ H^1_{(0)}$ is called the {\it reduced}
 Kontsevich-Zorich cocycle and is denoted by $G_t^{KZ,red}$.
 \end{definition}
 \medskip

\begin{remark}In general, $H_1(M,\mathbb{R})$ \emph{doesn't} have a $\textrm{Aff}(M,\omega)$-invariant supplement inside $H_1(M,\Sigma,\mathbb{R})$: see Appendix~\ref{a.supplement}.
\end{remark}

\subsection{Veech surfaces and square-tiled surfaces}

Let $(M,\omega)$ be a translation surface.
\begin{definition} $(M,\omega)$ is a {\it Veech surface} if $SL(M,\omega)$ is a lattice in $SL(2,\mathbb{R})$. This happens iff
the $SL(2,\mathbb{R})$-orbit of $(M,\omega)$ in $\mathcal{M}(M)$ is closed (see~\cite{HuSc} and~\cite{Zorich3}).
\end{definition}

The stabilizer of this $SL(2,\mathbb{R})$-orbit in $\Gamma(M)$ is exactly the affine group $\textrm{Aff}(M,\omega)$. We can thus view the reduced  Kontsevich-Zorich cocycle $G_t^{KZ,red}$ over this closed $SL(2,\mathbb{R})$-orbit as the quotient of the trivial cocycle 
$$\textrm{diag}(e^t,e^{-t}) \times id \,:SL(2,\mathbb R) \times H^1_{(0)}\to SL(2,\mathbb R) \times H^1_{(0)}$$
by the action of the affine group $\textrm{Aff}(M,\omega)$.\\

The two examples that we will consider belong to a special kind of Veech surfaces.

\begin{definition} $(M,\omega)$ is a 
{\it square-tiled surface} if the integral of $\omega$ over any path joining two zeros of $\omega$ belongs to $\mathbb{Z} + i\mathbb{Z}$. 
Equivalently, there exists a ramified covering $\pi: M \rightarrow \mathbb{R}^2/\mathbb{Z}^2$ unramified outside $0 \in \mathbb{R}^2/\mathbb{Z}^2$
such that $\omega = \pi^*(dz)$. Every square-tiled surface is a Veech surface. One says that the square-tiled surface $(M,\omega)$ is {\it primitive} if the relative periods of $\omega$ 
span the $ \mathbb{Z}$-module
$\mathbb{Z} + i\mathbb{Z}$. In this case the Veech group $SL(M,\omega)$ is a subgroup of $SL(2,\mathbb{Z})$ of finite index (see~\cite{HuSc},~\cite{Zorich3}).
\end{definition}

In a square-tiled surface $(M, \omega)$, the squares are the connected components of the inverse image $\pi^{-1}((0,1)^2)$, with 
$\pi: M \rightarrow \mathbb{R}^2/\mathbb{Z}^2$ as above. The set ${\rm Sq}(M,\omega)$ of squares of $(M,\omega)$ is finite and equipped with two one-to-one self maps 
$r$ (for right) and $u$ (for up) which associate
to a square the square to the right of it (resp. above it). The connectedness of the surface means that the group of permutations of 
${\rm Sq}(M,\omega)$ generated by $r$ and $u$ acts transitively on ${\rm Sq}(M,\omega)$. Conversely, a finite set $S$, equipped
with two one-to-one maps 
$r$  and $u$ such that the group of permutations 
 generated by $r$ and $u$ acts transitively on $S$, defines a square-tiled surface. See~\cite{Zorich3}.
 
 \medskip
 
 For a square-tiled surface, it is easy to identify the factors $H_1^{st}$ and $H_1^{(0)}$ in the decomposition of the homology groups, and to see that in this case they are defined over $\mathbb{Q}$. 
 
 Let $\Sigma'\supset\Sigma$ be the inverse image of $\{0\}$ under the ramified covering $\pi$. For each square $i \in {\rm Sq}(M,\omega)$, let $\sigma_i \in  H_1(M,\Sigma',\mathbb{Z})$ be the homology class defined by a path in $i$ from the bottom left corner to the bottom right corner; let $\zeta_i $ be the homology class defined by a path in $i$ from the bottom left corner to the upper left corner. Let $\sigma$ (resp. $\zeta$) be the sum over ${\rm Sq}(M,\omega)$ of the $\sigma_i$ (resp. of the $\zeta_i $). It is clear that both $\sigma$ and $\zeta$ belong to $H_1(M,\mathbb{Z})$. Let $\widetilde{\sigma}$ (resp., $\widetilde{\zeta}$) be the class in $H_1(M-\Sigma,\mathbb{Z})$ obtained from $\sigma$ (resp., $\zeta$) by shifting each $\sigma_i$ (resp., $\zeta_i$) slightly upwards (resp., to the right).
 
 \begin{proposition}\label{p.1}\begin{enumerate}
 \item The subspace $H_1^{(0)}$ (resp., $H_1^{(0)}(M,\Sigma,\mathbb{R})$) is the kernel of the homomorphism from $H_1(M,\mathbb{R})$ (resp., $H_1(M,\Sigma,\mathbb{R})$) to $H_1(\mathbb{R}^2/\mathbb{Z}^2,\mathbb{R})$ (resp., $H_1(\mathbb{R}^2/\mathbb{Z}^2,\{0\},\mathbb{R})$) induced by the ramified covering $\pi$.
 \item One has $$H_1^{st}(M-\Sigma,\mathbb{R}) = \mathbb{R}\widetilde{\sigma}\oplus\mathbb{R}\widetilde{\zeta},$$
 $$H_1^{st} = \mathbb{R}\sigma\oplus\mathbb{R}\zeta.$$
 Moreover, the action of the affine group on $H_1^{st}$ is through the homomorphism from the affine group to $SL(M,\omega)\subset SL(2,\mathbb{R})$ and the standard action of $SL(2,\mathbb{R})$ on $\mathbb{R}\sigma\oplus\mathbb{R}\zeta$.
 \end{enumerate}
 \end{proposition}
 
 \begin{proof}The first part is an immediate consequence of the definitions of $H_1^{(0)}$ and $H_1^{(0)}(M,\Sigma,\mathbb{R})$. 
 
 Let $\theta = \sum a_i\sigma_i+\sum b_i\zeta_i$ be a class in $H_1(M,\Sigma,\mathbb{R})$. For the intersection form between $H_1(M-\Sigma,\mathbb{R})$ and $H_1(M,\Sigma,\mathbb{R})$, we have $$\langle\widetilde{\sigma},\theta\rangle = \sum b_i,$$
 $$\langle\widetilde{\zeta},\theta\rangle = \sum a_i.$$
 
 On the other hand, we have $\theta\in H_1^{(0)}(M,\Sigma,\mathbb{R})$ iff $\sum a_i=\sum b_i = 0$. This shows that $\widetilde{\sigma}$ and $\widetilde{\zeta}$ belong to $H_1^{st}(M-\Sigma,\mathbb{R})$. As this subspace is $2$-dimensional and $\widetilde{\sigma}$, $\widetilde{\zeta}$ are linearly independent, we conclude that $H_1^{st}(M-\Sigma,\mathbb{R}) = \mathbb{R}\widetilde{\sigma}\oplus\mathbb{R}\widetilde{\zeta}$. It follows that $H_1^{st} = \mathbb{R}\sigma\oplus\mathbb{R}\zeta$. The last assertion of the proposition follows from a direct verification.
 \end{proof}
 
 We will denote by $H_1^{(0)}(M,\Sigma,\mathbb{Q})$ the kernel of the homomorphism induced by $\pi$ from $H_1(M,\Sigma,\mathbb{Q})$ to $H_1(\mathbb{R}^2/\mathbb{Z}^2,\{0\},\mathbb{Q})$. We have $H_1^{(0)}(M,\Sigma,\mathbb{R}) = \mathbb{R}\otimes_{\mathbb{Q}}H_1^{(0)}(M,\Sigma,\mathbb{Q})$. For $H_1^{(0)}$ and $H_1^{st}$, we omit the coefficients to keep the notation simple: $H_1^{st}$ designates $\mathbb{Q}\sigma\oplus\mathbb{Q}\zeta$ or $\mathbb{R}\sigma\oplus\mathbb{R}\zeta$ (the context should remove the ambiguity); $H_1^{(0)}$ is the kernel of the homomorphism between the first absolute homology groups of $M$ and $\mathbb{R}^2/\mathbb{Z}^2$ with real or rational coefficients.
 
\subsection{Degenerate $SL(2,\mathbb{R})$-orbits}

 Veech has asked how ``\emph{degenerate}'' the
 Lyapunov spectrum of $G_t^{KZ}$ can be along a \emph{non-typical} $SL(2,\mathbb{R})$-orbit, for instance along a closed orbit.

\medskip
 
This question was first answered by 
G. Forni~\cite{ForniSurvey} who exhibited a beautiful example of a square-tiled  surface $(M_3,\omega_{(3)})$ of genus $g=3$ such that the Lyapunov exponents of $G_t^{KZ}$ for the $SL(2,\mathbb{R})$
-invariant measure $\mu$ supported on the $SL(2,\mathbb{R})$-orbit of $(M_3,\omega_{(3)})$ verify $\lambda_2^\mu=\lambda_3^\mu=0$. 

\medskip

Subsequently,
 Forni and Matheus~\cite{FM} constructed a square-tiled surface $(M_4,\omega_{(4)})$  of genus 4 such that the Lyapunov exponents (with respect to $G_t^{KZ}$) of the $SL(2,\mathbb{R})$-invariant
 measure $\mu$ supported on the $SL(2,\mathbb{R})$-orbit of $(M_4,\omega_{(4)})$ verify $\lambda_2^\mu=\lambda_3^\mu=\lambda_4^\mu=0$.
 
 

 
More precisely, Forni's  example $(M_3,\omega_{(3)})$ is the Riemann surface of genus 3
\begin{equation}\label{e.Forni}
M_3=M_3(x_1,x_2,x_3,x_4)=\{(z,w): w^4 = \prod\limits_{\mu=1}^4(z-x_\mu)\}
\end{equation}
equipped with the Abelian differential $\omega_{(3)} = dz/w^2$. This example was independently discovered, for different reasons, by Herrlich, M\"oller and Schmith\"usen~\cite{HS}.

\medskip

Similarly, Forni and Matheus' example $(M_4,\omega_{(4)})$ is the Riemann surface of genus 4
\begin{equation}\label{e.ForniM}
M_4=M_4(x_1,x_2,x_3)=\{(z,w): w^6 = z^3\cdot\prod\limits_{\mu=1}^3(z-x_\mu)\}
\end{equation}
equipped with the Abelian differential $\omega_{(4)} = zdz/w^3$.

\begin{remark}Actually, these formulas for $(M_3,\omega_{(3)})$ and $(M_4,\omega_{(4)})$ are the description of \emph{entire} closed $SL(2,\mathbb{R})$-orbits. The corresponding square-tiled surfaces (in the sense of the previous definition) belonging to these orbits are obtained by appropriate choices of the points $x_\mu$. Also, we point out that these examples are particular cases of a more general family studied by I.~Bouw and M.~M\"oller~\cite{BM}.
\end{remark}

\begin{remark}In the sequel, $\Sigma_{(3)}$ denotes the set of zeroes of $\omega_{(3)}$ and $\Sigma_{(4)}$ denotes the set of zeroes of
 $\omega_{(4)}$. Note that $\#\Sigma_{(3)}=4$ and $\#\Sigma_{(4)}=3$, so that $\omega_{(3)}$ is an Abelian differential in the stratum
 $\mathcal M_{(1,1,1,1)}$ and $\omega_{(4)}$ is an Abelian differential in the stratum\footnote{This stratum has two connected components
 distinguished by the parity of the spin structure (see~\cite{KZ}). In particular, one can ask about the connected component of
 Forni-Matheus's surface. In the Appendix~\ref{a.g4parity} below, we'll use a square-tiled representation of this example to show that its
 spin structure is \emph{even}. } $\mathcal M_{(2,2,2)}$.
\end{remark}

\begin{remark}An unpublished work of Martin M\"oller~\cite{Moller} indicates that such examples with totally degenerate KZ spectrum are
 very rare: they don't exist in genus $g\geq 6$, Forni's example is the unique totally degenerate $SL(2,\mathbb{R})$-orbit in genus 3
 and the Forni-Matheus example is the unique totally degenerate $SL(2,\mathbb{R})$-orbit in genus 4; also, the sole stratum in genus 5
 \emph{possibly} supporting a totally degenerate $SL(2,\mathbb{R})$-orbit is $\mathcal M_{(2,2,2,2)}$, although this isn't
 \emph{probably} the case
 (namely, M\"oller pursued a computer program search and it seems that the possible exceptional case $\mathcal M_{(2,2,2,2)}$ can be
 ruled out).
\end{remark}

\begin{remark}
Forni's version of Kontsevich formula for the sum of the Lyapunov exponents reveals the following interesting feature of the
 Kontsevich-Zorich cocycle over a totally degenerate $SL(2,\mathbb{R})$: it is \emph{isometric} with respect to the \emph{Hodge norm}
 on the cohomology $H^1(M,\mathbb{R})$ on the orthogonal complement of the subspace associated to the exponents $\pm 1$. For more details
 see~\cite{Forni} and~\cite{ForniSurvey}. Observe that this fact is far from trivial \emph{in general} since the presence of zero
 Lyapunov exponents only indicates a subexponential (e.g., polynomial) divergence of the orbits (although in the specific case of the
 KZ cocycle, Forni manages to show that this ``subexponential divergence'' suffices to conclude there is no divergence at all).
\end{remark}

\subsection{Statement of the results}

We start with $(M_{3},\omega_{(3)})$. For this square-tiled surface, the Veech group is the full group $SL(2,\mathbb Z)$ and the automorphism group is the $8$-element quaternion group $Q := \{ \pm 1, \pm i, \pm j, \pm k \}$: see F. Herrlich and G. Schmith\"usen~\cite{HS} (and also Figure~\ref{Forni-origami} below). We have tried to summarize the main conclusions of the computations of the next section.

\begin{theorem}\label{t.A}
\begin{enumerate}
\item One has a decomposition
$$H_1(M_3,\Sigma_{(3)},\mathbb{Q})= H_1^{st} \oplus H_1^{(0)} \oplus H_{rel}$$ 
into $\mathbb Q$-defined $\textrm{Aff}(M_{3},\omega_{(3)})$-invariant subspaces. The action of $\textrm{Aff}(M_{3},\omega_{(3)})$ on $H_{rel}$ is through the group $S_4$ of permutations of the zeros of $\omega_{(3)}$.
\item There exists a root system $R$ of $D_4$ type spanning $H_1^{(0)}$ which is invariant under the action of $\textrm{Aff}(M_{3},\omega_{(3)})$. The action of $\textrm{Aff}(M_{3},\omega_{(3)})$ on $H_1^{(0)}$ is thus given by a homomorphism $Z$ of $\textrm{Aff}(M_{3},\omega_{(3)})$ to the automorphism group $A(R)$ of $R$. The image of this homomorphism is a subgroup of $A(R)$ of order $96$.
\item The inverse image in $\textrm{Aff}(M_{3},\omega_{(3)})$ of the Weyl group $W(R)$ is equal to the inverse image of the principal congruence subgroup $\Gamma(2)$ by the canonical morphism from $\textrm{Aff}(M_{3},\omega_{(3)})$ to $SL(M_{3},\omega_{(3)}) = 
SL(2,\mathbb{Z})$. The morphism from $\textrm{Aff}(M_{3},\omega_{(3)})$  to $A(R) / W(R)=S_3$ induced by $Z$ is onto.
\item The intersection of the image of $Z$ with $W(R)$ is the subgroup of order $16$ formed by those elements of $W(R)$ which preserve the intersection form on $H_1^{(0)}$.
\item The intersection of the kernel of $Z$ with the kernel of the action of $\textrm{Aff}(M_{3},\omega_{(3)})$ on $H_{rel}$ is sent isomorphically onto the principal congruence subgroup $\Gamma(4)$ by the canonical morphism from $\textrm{Aff}(M_{3},\omega_{(3)})$ to $ SL(2,\mathbb{Z})$.
\end{enumerate}
\end{theorem}

\begin{remark} Avila and Hubert communicated to the authors that they  checked that the action of generators of $\textrm{Aff}(M_{3},\omega_{(3)})$ on 
$H_1^{(0)}$ was through matrices of finite order. One of the referees also brought our attention to the PhD thesis of Oliver Bauer~\cite{Bauer}, who does some computations on the action of the affine group similar to ours.
\end{remark}

We now consider $(M_{4},\omega_{(4)})$. For this square-tiled surface, we will see that the Veech group is the full group $SL(2,\mathbb Z)$ and the automorphism group is the cyclic group $\mathbb Z /3$.

\begin{theorem}\label{t.B}
\begin{enumerate}
\item One has  decompositions
$$H_1(M_4,\Sigma_{(4)},\mathbb{Q})= H_1^{st} \oplus H_1^{(0)} \oplus H_{rel},$$ 
$$H_1^{(0)} = H_{\tau} \oplus \breve H$$
into  $\mathbb Q$-defined $\textrm{Aff}(M_{4},\omega_{(4)})$-invariant subspaces. The action of $\textrm{Aff}(M_{4},\omega_{(4)})$ on $H_{rel}$ is through the group $S_3$ of permutations of the zeros of $\omega_{(4)}$.
\item The subspace $H_{\tau}$ is $2$-dimensional and the action of $\textrm{Aff}(M_{4},\omega_{(4)})$ on it is through a homomorphism to the cyclic group $\mathbb Z /6$ (acting by rotations).
\item The subspace $\breve H$ is $4$-dimensional and it splits over $\mathbb C$ into two $\textrm{Aff}(M_{4},\omega_{(4)})$-invariant subspaces of dimension 2.
\item There exists a root system $R$ of $D_4$ type spanning $\breve H$ which is invariant under the action of $\textrm{Aff}(M_{4},\omega_{(4)})$. The action of $\textrm{Aff}(M_{4},\omega_{(4)})$ on $\breve H$ is thus given by a homomorphism $Z$ of $\textrm{Aff}(M_{4},\omega_{(4)})$ to the automorphism group $A(R)$ of $R$. The image of this homomorphism is a subgroup of $A(R)$ of order $72$.
\item The inverse image in $\textrm{Aff}(M_{4},\omega_{(4)})$ of the Weyl group $W(R)$ is sent isomorphically onto the Veech group $SL(M_{4},\omega_{(4)}) = SL(2,\mathbb{Z})$ by the canonical morphism from the affine group to the Veech group. The image of the morphism from $\textrm{Aff}(M_{4},\omega_{(4)})$  to $A(R) / W(R)=S_3$ is the cyclic subgroup of index $2$.
\item The intersection of the image of $Z$ with $W(R)$ is the subgroup of order $24$, isomorphic to $SL(2,\mathbb Z /3)$, formed by those elements of $W(R)$ which preserve the intersection form on $\breve H$.
\item The kernel of $Z$ is sent isomorphically onto the principal congruence subgroup $\Gamma(3)$ by the canonical morphism from the affine group to the Veech group.

\end{enumerate}
\end{theorem}

Actually, in Section 3, we study a family of square-tiled surfaces parametrized by an odd integer $q\geq 3$, the surface $(M_{4},\omega_{(4)})$ corresponding to $q=3$. Some of the computations are valid for all $q\geq 3$, but the stronger statements only hold for $q=3$.\\

In view of the observation following Definition 1.8 (in Subsection 1.3), an immediate consequence of the theorems is that, for the two 
square-tiled surfaces considered above, all Lyapunov exponents of the reduced Kontsevich-Zorich cocycle are equal to zero.

Conversely (and less trivially), M\"oller has shown (\cite{Moller}) that the action on $H_1^{(0)}$ for a totally degenerate square-tiled surface has to be through a finite group.

\subsection*{Acknowledgements} This research has been supported by the following institutions: the Coll\`ege de France, French ANR (grants 0863 petits diviseurs et resonances en g\'eom\'etrie, EDP et dynamique 0864 Dynamique dans l'espace de Teichm\"uller). We thank the Coll\`ege de France, IMPA (Rio de Janeiro), the Max-Planck Institute f\"ur Mathematik in Bonn and the Mittag-Leffler Institute for their hospitality. We are also grateful to the referees for suggestions that greatly improved the presentation of the paper.


\section{Proof of Theorem~\ref{t.A}}

This section is organized as follows. In Subsection 2.1, we recall the description of $(M_3,\omega_{(3)})$ as a square-tiled surface. Also, the automorphism group is identified with the quaternion group. In order to understand the action of this group on the homology, we recall in Subsection 2.2 the list of irreducible representations of the quaternion group. In Subsection 2.3, we introduce generators for $H_1(M_3,\Sigma_{(3)},\mathbb{Z})$, compute the action of the quaternion group on homology and break the homology into invariant subspaces. Generators for the affine group are chosen in Subsection 2.4, and their action on homology are computed in Subsection 2.5. This allows to identify in Subsection 2.6 a subspace $H_{rel}$ which complements $H_1(M_3,\mathbb{Q})$ in $H_1(M_3,\Sigma_{(3)},\mathbb{Q})$ and is invariant under the action of the affine group. Finally, the action of the affine group on $H_1^{(0)}$ is analyzed in Subsection 2.7. 

\subsection{The square-tiled surface $(M_3,\omega_{(3)})$}

We follow here F. Herrlich and G. Schmith\"usen \cite{HS}. The set $Sq(M_3,\omega_{(3)})$ is identified with the quaternion group $Q=\{\pm 1, \pm i, \pm j, \pm k\}$. We denote by $sq(g)$ the square corresponding to $g \in Q$. The map $r$ (for right) is $sq(g) \mapsto sq(g i)$ and the map $u$ (for up) is $sq(g) \mapsto sq(g j)$. The automorphism group 
${\rm Aut} (M_3,\omega_{(3)})$ is then canonically identified with $Q$, the element $h \in Q$ sending the square $sq(g)$ on the square $sq(h g)$. See Figure~\ref{Forni-origami} below.

\begin{figure}[!h]
\includegraphics[scale=0.3]{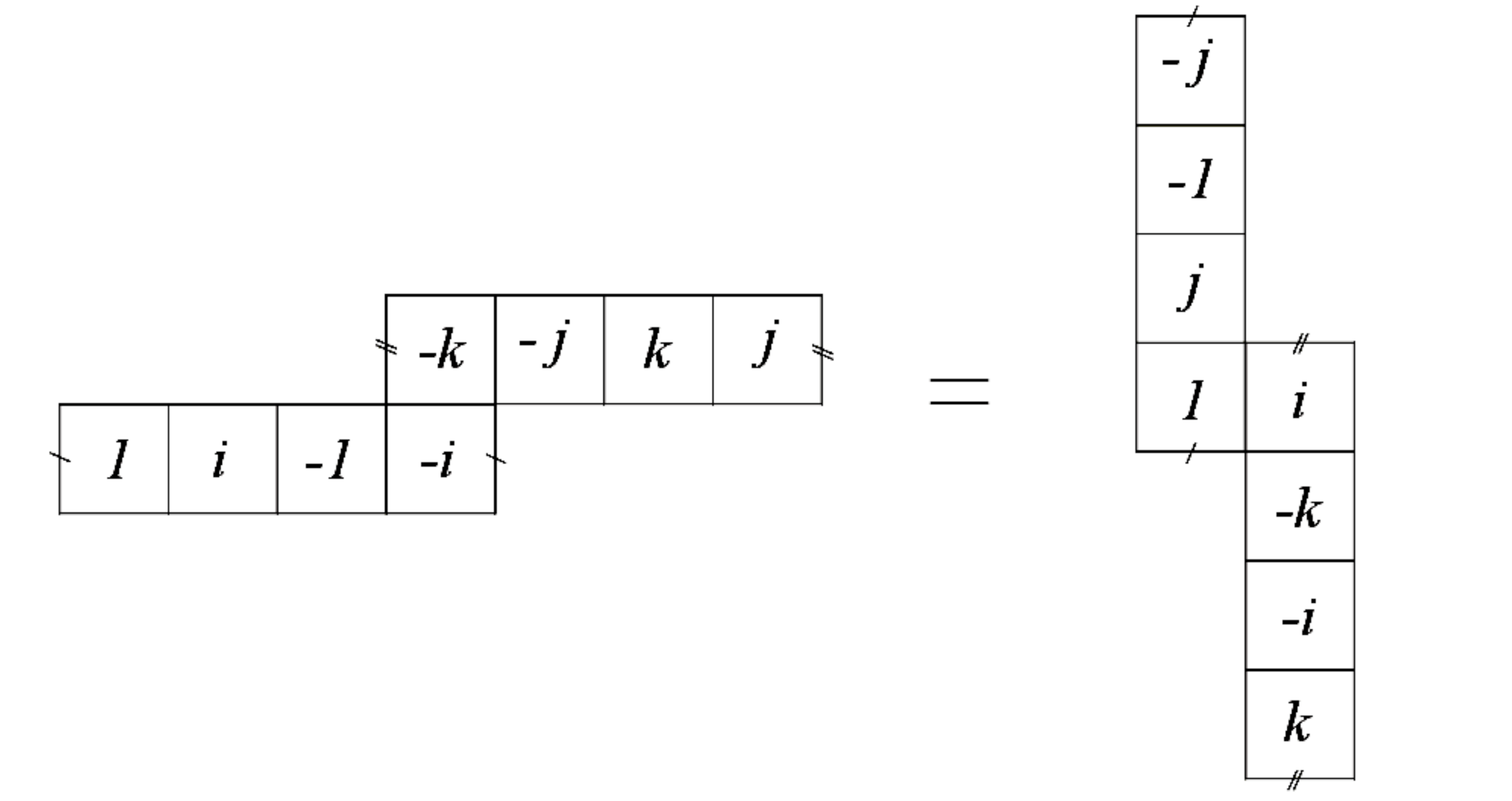}
\caption{Forni's \textit{Eierlegende Wollmilchsau}.}\label{Forni-origami}
\end{figure}

Here, it is shown the horizontal and vertical cylinder decompositions, and the right and top neighbors of each square, so that the (implicit) side identifications are easily deduced.

\medskip

We will denote by $\overline Q$ the quotient of $Q$ by its center $\{\pm 1 \}$. It is isomorphic to $\mathbb{Z}_2 \times \mathbb{Z}_2$. We denote
by $\overline 1, \;\overline i, \;\overline j, \;\overline k $ the images of $\pm 1, \pm i, \pm j, \pm k$ in $\overline Q$.

\medskip

For $g \in Q$, the lower left corners of $sq(g)$ and $sq(-g)$ correspond to the same point of $\Sigma_{(3)}$. We will identify in this way
$\Sigma_{(3)}$ with $\overline Q$.

\medskip

The map $\pi: M_3 \rightarrow \mathbb{R}^2/\mathbb{Z}^2$ factors as
$$M_3 \rightarrow \mathbb{R}^2/(2\mathbb{Z})^2 \rightarrow \mathbb{R}^2/\mathbb{Z}^2.$$
Here, the first map $\pi_1$ is a two fold covering ramified over the four points of order $2$ in $\mathbb{R}^2/(2\mathbb{Z})^2$, and may be viewed as the quotient map by the involution $-1 \in Q$ of $M_3$, the four squares of $\mathbb{R}^2/(2\mathbb{Z})^2$
being naturally labelled by $\overline Q$.

\subsection{Irreducible representations of $Q$}

The group $Q$ has 5 distinct irreducible representations, 4 one-dimensional $\chi_1,\,\chi_i,\,\chi_j,\,\chi_k$ and one $2$-dimensional $\chi_2$.
The character table is 

\begin{center}
{\renewcommand{\arraystretch}{1.5}
\renewcommand{\tabcolsep}{0.2cm}
\begin{tabular}{|c|c|c|c|c|c|}
\hline
 & $1$ & $-1$ & $\pm i$ & $\pm j$ & $\pm k$ \\
\hline
$\chi_1$ & $1$ & $1$ & $1$ & $1$ & $1$ \\
\hline
 $\chi_i$ & $1$ & $1$ & $1$ & $-1$ & $-1$\\
\hline
$\chi_j$ & $1$ & $1$ & $-1$ & $1$ & $-1$\\
\hline
$\chi_k$ & $1$ & $1$ & $-1$ & $-1$ & $1$\\
\hline
tr $\chi_2$ & $2$ & $-2$ & $0$ & $0$ & $0$\\
\hline
\end{tabular}}
\end{center}

\bigskip

In the regular representation of $Q$ in $\mathbb{Z}(Q)$, the submodules associated to these representations are generated by

\medskip

\begin{center}
{\renewcommand{\arraystretch}{1.5}
\renewcommand{\tabcolsep}{0.2cm}
\begin{tabular}{|c|c|}

\hline
$\chi_1$ & $[1]+[-1]+[i]+[-i]+[j]+[-j] +[k]+[-k]$ \\
\hline
 $\chi_i$ & $[1]+[-1]+[i]+[-i]-[j]-[-j] -[k]-[-k]$\\
\hline
$\chi_j$ & $[1]+[-1]-[i]-[-i]+[j]+[-j] -[k]-[-k]$\\
\hline
$\chi_k$ & $[1]+[-1]-[i]-[-i]-[j]-[-j] +[k]+[-k]$\\
\hline
$\chi_2$ & $[1]-[-1]$, $\;[i]-[-i]$, $\;[j]-[-j]$, $\; [k]-[-k]$\\
\hline
\end{tabular}}
\end{center}

\bigskip

\subsection{Action of  ${\rm Aut} (M_3,\omega_{(3)})$ on $H_1(M_3,\Sigma_{(3)},\mathbb{Z})$} 

We consider the direct sum $\mathbb{Z}(Q) \oplus \mathbb{Z}(Q)$ of two copies of $\mathbb{Z}(Q)$. We denote by $(\sigma_g)_{g \in Q}$ the canonical
basis of the first copy and by $(\zeta_g)_{g \in Q}$ the canonical
basis of the second copy. We define a homomorphism $p$ from $\mathbb{Z}(Q) \oplus \mathbb{Z}(Q)$ onto  $H_1(M_3,\Sigma_{(3)},\mathbb{Z})$ by sending
$\sigma_g$ on the homology class defined by the lower side of $sq(g)$ (oriented from left to right) and $\zeta_g$ on the homology class defined by the left side of $sq(g)$ (oriented from upwards)
 (see Figure~\ref{Forni-coordinates} below). The homomorphism $p$ is compatible with the actions of $Q$ on $\mathbb{Z}(Q) \oplus \mathbb{Z}(Q)$ (by the regular representation)
 and on $H_1(M_3,\Sigma_{(3)},\mathbb{Z})$ (identifying $Q$ with ${\rm Aut} (M_3,\omega_{(3)})$).
 
 \medskip


\begin{figure}[!h]
\centering
\includegraphics[scale=0.15]{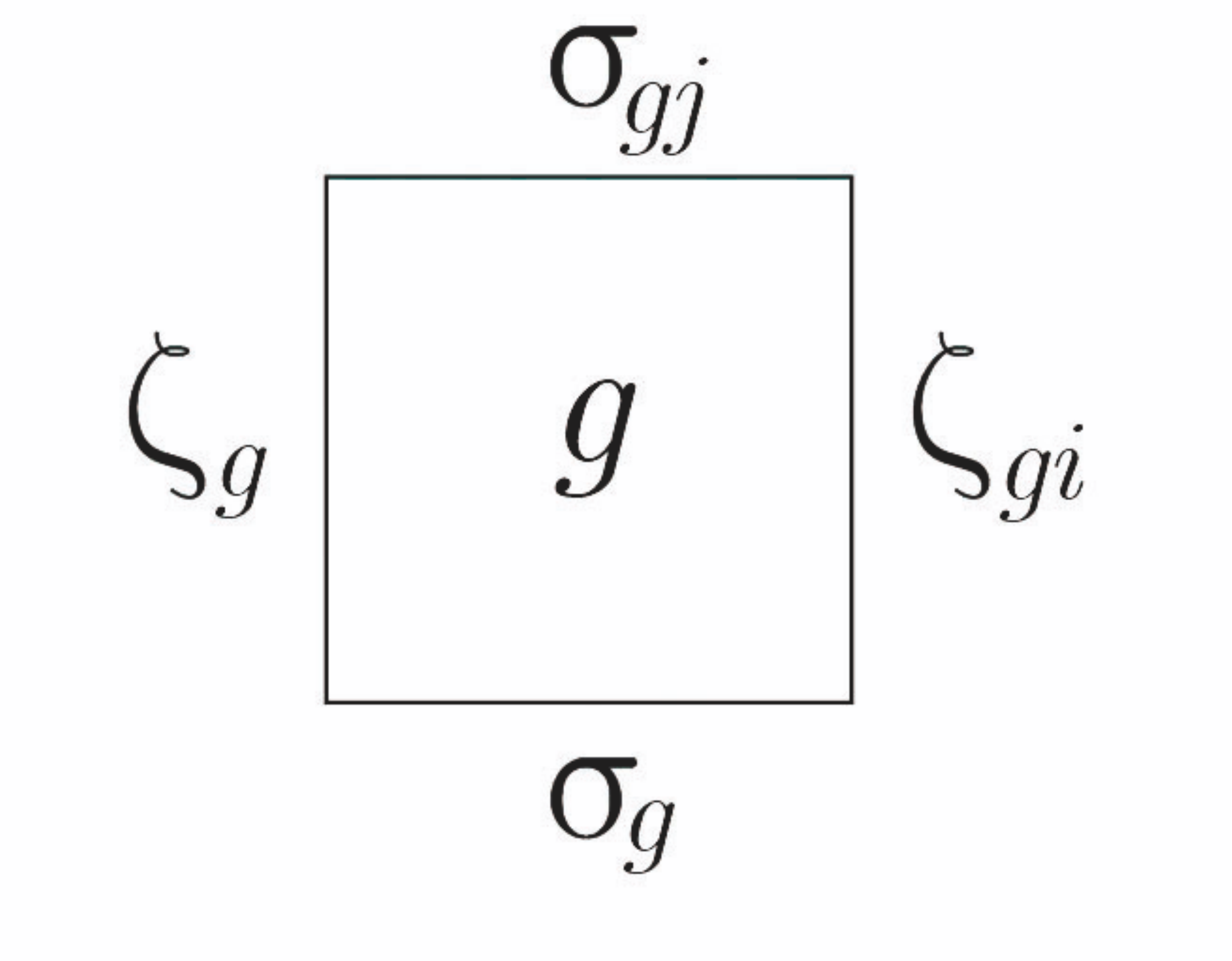}
\caption{Generators and relations for $H_1(M_{(3)},\Sigma_{(3)},\mathbb{Z})$.}\label{Forni-coordinates}
\end{figure}


The kernel of the homomorphism $p$ is the submodule $Ann$ of $\mathbb{Z}(Q) \oplus \mathbb{Z}(Q)$ generated by the elements
\begin{equation}\label{e.relations-g3}
\Box_g :=\sigma_g+\zeta_{gi} - \zeta_g-\sigma_{gj}.
\end{equation}

We have $\sum_Q \Box_g =0$, hence $Ann$ has rank $7$. Observe that, for $g,h \in Q$ we have $h.\Box_g = \Box_{hg}$.

\bigskip

Recall that $\Sigma_{(3)}$ is identified with $\overline Q$. The boundary map $\partial: H_1(M_{(3)},\Sigma_{(3)},\mathbb{Z}) \rightarrow
\mathbb{Z}(\overline Q)$ is induced by
\begin{eqnarray*}
\partial (p(\sigma_g)) &=& \overline {gi} - \overline g, \\
\partial (p(\zeta_g)) &=& \overline {gj} - \overline g.
\end{eqnarray*}
\medskip
Let
\begin{eqnarray*}
w_i &=& p(\zeta_1 + \zeta_{-1} + \zeta_i + \zeta_{-i} -\zeta_j -\zeta_{-j} -\zeta_k -\zeta_{-k}), \\
w_j &=& p(\sigma_1 +\sigma_{-1} +\sigma_j +\sigma_{-j}-\sigma_i -\sigma_{-i}-\sigma_k -\sigma_{-k}), \\
w_k &=& p(\zeta_1 + \zeta_{-1} + \zeta_k + \zeta_{-k} -\zeta_j -\zeta_{-j} -\zeta_i -\zeta_{-i}), \\
 &=& p(\sigma_1 +\sigma_{-1} +\sigma_k +\sigma_{-k}-\sigma_i -\sigma_{-i}-\sigma_j -\sigma_{-j}). \\
\end{eqnarray*}
We note that
\begin{eqnarray*}
0 &=& p(\zeta_1 + \zeta_{-1} + \zeta_j + \zeta_{-j} -\zeta_i -\zeta_{-i} -\zeta_k -\zeta_{-k}), \\
  &=& p(\sigma_1 +\sigma_{-1} +\sigma_i +\sigma_{-i}-\sigma_j -\sigma_{-j}-\sigma_k -\sigma_{-k}). \\
\end{eqnarray*}

We have $\partial (w_i) = 4(\overline j + \overline k -\overline i -\overline 1)$ and 
\begin{equation}\label{e.rel-g3}
g.w_i = \chi_i(g) w_i
\end{equation}
\medskip
for $g \in Q$. Similar statements
hold for $w_j$ and $w_k$. Let $H_{rel}= \mathbb{Q}w_i \oplus \mathbb{Q}w_j \oplus \mathbb{Q}w_k$ be the subspace of  $H_1(M_3,\Sigma_{(3)},\mathbb{Q})$ spanned by
$w_i,\,w_j,\,w_k$. The formulas imply the following lemma:
\begin{lemma}The subspace $H_{rel}$ is invariant under the action of $\textrm{Aut}(M_3,\omega_{(3)})$ and complements $H_1(M_3,\mathbb{Q})$ in $H_1(M_3,\Sigma_{(3)},\mathbb{Q})$.
\end{lemma}

\begin{proof}The first part is clear. The second part follows from the fact that the images of $w_i,w_j,w_k$ under $\partial$ are linearly independent.
\end{proof}

Let
\begin{eqnarray*}
\sigma &=&p(\sum_Q \sigma_g),\\
\zeta &=& p(\sum_Q \zeta_g).
\end{eqnarray*}

From Subsection 1.3, we have $H_1^{st}=\mathbb{Q}\sigma\oplus\mathbb{Q}\zeta$.

We have $ \partial(\sigma) = \partial(\zeta) =0$ and 
\begin{eqnarray}\label{e.st-g3}
g.\sigma &=& \sigma,\\
 g.\zeta &=&\zeta, 
\end{eqnarray}
\medskip
for all $g \in Q$.

\bigskip

We set, for $g \in Q$,
\begin{eqnarray*}
\widehat{\sigma}_g &=& p(\sigma_g - \sigma_{-g}),\\
\widehat{\zeta}_g &=& p(\zeta_g - \zeta_{-g}),\\
\varepsilon_g &=& \widehat{\sigma}_g-\widehat{\sigma}_{gj} =\widehat{\zeta}_{g} - \widehat{\zeta}_{gi}
\end{eqnarray*}
We have $\partial (\widehat{\sigma}_g) = \partial (\widehat{\zeta}_g) =0$. The subspace  of $H_1(M_3, \mathbb{Q})$ generated by the
$\widehat{\sigma}_g, \;\widehat{\zeta}_g, \; g\in Q$ has rank $4$ and it is the kernel $H_1^{(0)}$ of the homomorphism $\pi_*: 
H_1(M_3, \mathbb{Q}) \rightarrow H_1 ( \mathbb{R}^2 / \mathbb{Z}^2, \mathbb{Q})$ (see Subsections 1.2, 1.3).

For $g \in Q$, we have 
\begin{eqnarray}\label{e.sym-g3}
\widehat{\sigma}_{-g} &=&- \widehat{\sigma}_g,\\
\widehat{\zeta}_{-g} &=&- \widehat{\zeta}_g,\\
 \varepsilon_{-g} &=& -\varepsilon_g,
 \end{eqnarray}
and 
\begin{eqnarray}\label{e.eps-g3}
\widehat{\sigma}_{g} &=&\frac 12 ( \varepsilon_{g} +  \varepsilon_{gj}),\\
\widehat{\zeta}_{g} &=&\frac 12 ( \varepsilon_{g} +  \varepsilon_{gi}).
\end{eqnarray} 

The subspace $H_1^{(0)}$ is $Q$-invariant, being (in many ways) sum of two copies of $\chi_2$. One has, for $g,h \in Q$,
\begin{eqnarray}\label{e.aut-g3}
h.\widehat{\sigma}_{g} = \widehat{\sigma}_{hg},\\
 h.\widehat{\zeta}_{g} = \widehat{\zeta}_{hg},\nonumber\\
h.\varepsilon_g = \varepsilon_{hg}.\nonumber
\end{eqnarray}

At this stage, we have written the homology group $H_1(M_3,\Sigma_{(3)},\mathbb{Q})$ as the direct sum $H_1^{st}\oplus H_1^{(0)}\oplus H_{rel}$ of $\textrm{Aut}(M_3,\omega_{(3)})$-invariant summands.

\bigskip

\subsection{The affine group ${\rm Aff}_{(1)} (M_3,\omega_{(3)})$}Let $\textrm{Aff}_{(1)}(M_3,\omega_{(3)})$ be the stabilizer of 
$\overline 1 \in \Sigma_{(3)}$ in $\textrm{Aff}(M_3,\omega_{(3)})$. 
\begin{lemma}The subgroup $\textrm{Aff}_{(1)}(M_3,\omega_{(3)})$ has index $4$ in $\textrm{Aff}(M_3,\omega_{(3)})$. Moreover, its intersection with $Q =\textrm{Aut}(M_3,\omega_{(3)})$ is the center $Z=\{\pm 1 \}$ of $Q$ and is contained in the center of $\textrm{Aff}(M_3,\omega_{(3)})$. The subgroup $\textrm{Aff}_{(1)}(M_3,\omega_{(3)})$ is thus a central extension 
$$1 \longrightarrow Z=\{\pm 1 \} \longrightarrow \textrm{Aff}_{(1)}(M_3,\omega_{(3)}) \longrightarrow SL(M_3,\omega_{(3)})=SL(2,\mathbb{Z})
\longrightarrow 1\;.$$
\end{lemma}

\begin{proof}The group $\textrm{Aff}(M_3,\omega_{(3)})$ acts transitively on $\Sigma_{(3)}$, hence $\textrm{Aff}_{(1)}(M_3,\omega_{(3)})$ has index $4$. The intersection of $\textrm{Aff}_{(1)}(M_3,\omega_{(3)})$ with $Q$ is clearly the center $Z=\{\pm1\}$ of $Q$. Let $f\in\textrm{Aff}(M_3,\omega_{(3)})$ and $g:=-1$; then, $f^{-1}g^{-1}fg$ belongs to $\textrm{Aut}(M_3,\omega_{(3)})$ and fixes each point of $\Sigma_{(3)}$, hence is equal to $id$ or $g$; but the second possibility would imply that $f^{-1}g^{-1}f=id$. Thus, $-1$ belongs to the center $\textrm{Aff}(M_3,\omega_{(3)})$.

Finally, the sequence is exact because the intersection of $Q$ with $\textrm{Aff}_{(1)}(M_3,\omega_{(3)})$ is precisely $Z$.
\end{proof}

\begin{remark}We will see below that the center of $\textrm{Aff}(M_3,\omega_{(3)})$ is cyclic of order $4$.
\end{remark}

\medskip

We choose as generators of $SL(2,\mathbb{Z})$  

$$S = \left( \begin{array}{cc} 1 & 0 \\  1 & 1 \\  \end{array} \right)  \quad \quad \textrm{and}\quad\quad T = \left( \begin{array}{cc} 1 & 1 \\  0 & 1 \\  \end{array} \right). $$

We denote by $\widetilde{S}$, $\widetilde{T}$ the elements of $\textrm{Aff}_{(1)}(M_3,\omega_{(3)})$ with derivatives $S$, $T$ respectively
such that $sq(1)$ intersects $\widetilde{S}(sq(1))$ and $\widetilde{T}(sq(1))$. We have 
$$(\widetilde{S} \widetilde{T}^{-1}\widetilde{S})^4 = -1$$
and therefore $\widetilde{S}$, $\widetilde{T}$ form a system of generators for $\textrm{Aff}_{(1)}(M_3,\omega_{(3)})$.

\subsection{The action of $\widetilde{S}$, $\widetilde{T}$ on $H_1(M_3,\Sigma_{(3)},\mathbb{Z})$}

One easily checks that
\begin{equation}\label{e.S1-g3}
\widetilde{S}(p(\zeta_g))=\left\{
\begin{array}{cc}
p(\zeta_g) & \text{if } g\in\{\pm1,\pm j\},\\
p(\zeta_{jg}) & \text{if } g\in\{\pm i,\pm k\},
\end{array} \right.
\end{equation}

\begin{equation}\label{e.S2-g3}
\widetilde{S}(p(\sigma_g))=\left\{
\begin{array}{cc}
p(\sigma_g + \zeta_{gi}) & \text{if } g\in\{\pm1,\pm j\},\\
p(\sigma_{jg} + \zeta_{gk}) & \text{if } g\in\{\pm i,\pm k\},
\end{array} \right.
\end{equation}

\begin{equation}\label{e.T1-g3}
\widetilde{T}(p(\sigma_g))=\left\{
\begin{array}{cc}
p(\sigma_g) & \text{if } g\in\{\pm1,\pm i\},\\
p(\sigma_{ig}) & \text{if } g\in\{\pm j,\pm k\},
\end{array} \right.
\end{equation}

\begin{equation}\label{e.T2-g3}
\widetilde{T}(p(\zeta_g))=\left\{
\begin{array}{cc}
p(\zeta_g + \sigma_{gj}) & \text{if } g\in\{\pm1,\pm i\},\\
p(\zeta_{ig} + \sigma_{-gk}) & \text{if } g\in\{\pm j,\pm k\}.
\end{array} \right.
\end{equation}

From these formulas, we deduce
\begin{eqnarray}\label{e.st-g3}
\widetilde{S}(\sigma) & =\sigma + \zeta, \hspace{2cm} \widetilde{S}(\zeta) & = \zeta,\\
 \widetilde{T}(\zeta) & =\sigma + \zeta, \hspace{2cm} \widetilde{T}(\sigma) & = \sigma,\nonumber
 \end{eqnarray}
 
 \bigskip
\noindent and also

\begin{alignat}{4}\label{rel-g3}
\widetilde{S}(w_i) & =w_k, \hspace{15mm} \widetilde{S}(w_j) & = w_j, \hspace{15mm} \widetilde{S}(w_k) & = w_i,\\ 
\widetilde{T}(w_i) & =w_i, \hspace{15mm} \widetilde{T}(w_j) & = w_k, \hspace{15mm} \widetilde{T}(w_k) & = w_j. \nonumber
\end{alignat}

\bigskip

Finally, we have

\begin{equation}\label{e.S3-g3}
\widetilde{S}(\wh\zeta_g)=\left\{
\begin{array}{cc}
\wh\zeta_g & \text{if } g\in\{\pm1,\pm j\},\\
\wh\zeta_{jg} & \text{if } g\in\{\pm i,\pm k\},
\end{array} \right.
\end{equation}

\begin{equation}\label{e.S4-g3}
\widetilde{S}(\wh\sigma_g)=\left\{
\begin{array}{cc}
\wh\sigma_g + \wh\zeta_{gi} & \text{if } g\in\{\pm1,\pm j\},\\
\wh\sigma_{jg} +\wh\zeta_{gk} & \text{if } g\in\{\pm i,\pm k\},
\end{array} \right.
\end{equation}

\begin{equation}\label{e.S5-g3}
\widetilde{S}(\varepsilon_g)=\left\{
\begin{array}{cc}
\frac 12 ( \varepsilon_{g} + \varepsilon_{gi} + \varepsilon_{gj} + \varepsilon_{gk})& \text{if } g\in\{\pm1,\pm j\},\\
\frac 12 ( \varepsilon_{g} - \varepsilon_{gi} - \varepsilon_{gj} + \varepsilon_{gk}) & \text{if } g\in\{\pm i,\pm k\},
\end{array} \right.
\end{equation}

\begin{equation}\label{e.T3-g3}
\widetilde{T}(\wh\sigma_g)=\left\{
\begin{array}{cc}
\wh\sigma_g & \text{if } g\in\{\pm1,\pm i\},\\
\wh\sigma_{ig} & \text{if } g\in\{\pm j,\pm k\},
\end{array} \right.
\end{equation}

\begin{equation}\label{e.T4-g3}
\widetilde{T}(\wh\zeta_g)=\left\{
\begin{array}{cc}
\wh\zeta_g + \wh\sigma_{gj} & \text{if } g\in\{\pm1,\pm i\},\\
\wh\zeta_{ig} + \wh\sigma_{-gk} & \text{if } g\in\{\pm j,\pm k\},
\end{array} \right.
\end{equation}

\begin{equation}\label{e.T5-g3}
\widetilde{T}(\varepsilon_g)=\left\{
\begin{array}{cc}
\frac 12 ( \varepsilon_{g} + \varepsilon_{gi} + \varepsilon_{gj} - \varepsilon_{gk})& \text{if } g\in\{\pm1,\pm i\},\\
\frac 12 ( \varepsilon_{g} - \varepsilon_{gi} - \varepsilon_{gj} - \varepsilon_{gk}) & \text{if } g\in\{\pm j,\pm k\}.
\end{array} \right.
\end{equation}

\bigskip


\subsection{The subspaces $H_1^{st}$ and $H_{rel}$}

As mentioned in Subsection 1.3,
the action on $H_1^{st}$ of the affine group is through the homomorphism from the affine group to $SL(M_3,\omega_{(3)})=SL(2,\mathbb{Z})$ 
and the standard action of  $SL(2,\mathbb{Z})$ on $\mathbb{Q}\sigma \oplus \mathbb{Q}\zeta$.

\bigskip

The formulas \eqref{rel-g3} show that $H_{rel}$ is invariant under $\widetilde{S}$ and $\widetilde{T}$.

We introduce the congruence subgroup of $SL(2,\mathbb{Z})$
$$\Gamma(2):=\left\{M\equiv \textrm{Id}_{SL(2,\mathbb{Z})} \, \textrm{mod}\, 2\right\},$$
and denote by $\widetilde \Gamma(2)$ the inverse image of $\Gamma(2)$ in ${\rm Aff}_{(1)} (M_3,\omega_{(3)})$. We introduce also 
\begin{eqnarray*}
\wh w(\overline 1) &=& w_i+w_j+w_k,\\
\wh w(\overline i) &=& w_i-w_j-w_k,\\
\wh w(\overline j) &=& -w_i+w_j-w_k,\\
\wh w(\overline k) &=& -w_i-w_j+w_k.
\end{eqnarray*}
These four points are the vertices of a regular tetrahedron in $H_{rel}$.

\begin{lemma}\begin{enumerate}
\item The sequence $$1 \longrightarrow \widetilde \Gamma(2)  \longrightarrow {\rm Aff} (M_3,\omega_{(3)}) \longrightarrow S_4 \longrightarrow 1$$ is exact, where the morphism from ${\rm Aff} (M_3,\omega_{(3)})$ to $S_4$ is through the action of ${\rm Aff} (M_3,\omega_{(3)})$ on $\Sigma_{(3)}$.
\item The subspace $H_{rel}$ is invariant under the action of ${\rm Aff} (M_3,\omega_{(3)})$. The subgroup $\widetilde \Gamma(2)$ is exactly the kernel of the representation of
${\rm Aff} (M_3,\omega_{(3)})$ in $GL(H_{rel},\mathbb{Q})$. The affine group acts on $H_{rel}$ through the symmetry group of the tetrahedron with vertices $\{\wh w(\overline 1), \wh w(\overline i), \wh w(\overline j), \wh w(\overline k)\}$.
\end{enumerate}
\end{lemma}

\begin{proof}We first prove that  the homomorphism from ${\rm Aff} (M_3,\omega_{(3)})$ to $S_4$ is onto. Indeed, $\widetilde S$ fixes $\overline 1$ and $\overline j$ and exchanges $\overline i$ and $\overline k$, while $\widetilde T$ fixes $\overline 1$ and $\overline i$ and exchanges $\overline j$ and $\overline k$; therefore, ${\rm Aff}_{(1)}(M_3,\omega_{(3)})$ acts by the full permutation group of $\overline i$, $\overline j$, $\overline k$. As $Q=\textrm{Aut}(M_3,\omega_{(3)})$ acts transitively on $\Sigma_{(3)}$, the claim follows. 

We now prove that the kernel of the morphism from ${\rm Aff} (M_3,\omega_{(3)})$ to $S_4$ is $\widetilde\Gamma(2)$. Indeed, $\widetilde{S}^2$, $\widetilde{T}^2$ and $(\widetilde{S}\widetilde{T}^{-1} \widetilde S)^2$ act trivially on $\Sigma_{(3)}$. These elements generate $\widetilde\Gamma(2)$, which is therefore contained in the kernel. On the other hand, the index of $\widetilde\Gamma(2)$ in ${\rm Aff} (M_3,\omega_{(3)})$ is $24$ which is also the cardinality of $S_4$. This proves the first part of the lemma.

We have already noted that $H_{rel}$ is invariant under $\widetilde S$ and $\widetilde T$. As ${\rm Aff} (M_3,\omega_{(3)})$ is generated by $Q$ and these two elements, the first assertion of the second part follows. In fact, the formulas \eqref{e.rel-g3} and \eqref{rel-g3} show that $\widetilde S$, $\widetilde T$ and $Q$ preserve the set $\{\wh w(\overline 1), \wh w(\overline i), \wh w(\overline j), \wh w(\overline k)\}$ and act on this set as on $\Sigma_{(3)}$. This completes the proof of the lemma.
\end{proof}

It is worth mentioning that the group $S_4$ through which ${\rm Aff} (M_3,\omega_{(3)})$ acts on $H_{rel}$ is the Weyl group $W(A_3)$
of an $A_3$ root system.
\bigskip

\subsection{The action of $\textrm{Aff}(M_3,\omega_{(3)})$ on $H_1^{(0)}$}

In order to get a nice description of this action, we briefly recall the main features of $D_l$ root systems (see~\cite{Bourbaki} for proofs and further details). They are typically obtained from Lie algebras isomorphic to $\mathfrak{so}(2l)$.

\medskip

Let $l\geq 4$ and $V$ be a $l$-dimensional space equipped with a scalar product; let $(\varepsilon_m)_{1\leq m\leq l}$ be an orthonormal basis of $V$.

The set $R=\{\pm\varepsilon_m\pm\varepsilon_n, 1\leq m<n\leq l\}$ is a \emph{root system of type} $D_l$. 

For every $\alpha\in R$, let $s_{\alpha}$ be the orthogonal symmetry with respect to the hyperplane orthogonal to $\alpha$. Then, $s_\alpha(R)=R$. The subgroup of $O(V)$ generated by the $s_{\alpha}$ is the \emph{Weyl group} $W(R)$. It is isomorphic to the semi-direct product of the symmetric group $S_l$ (acting by permutation of the $\varepsilon_m$) by the group $(\mathbb Z/2)^{l-1}$ (acting by $\varepsilon_m\mapsto(\pm1)_m\cdot\varepsilon_m$ with $\prod\limits_{m}(\pm1)_m=1$).

The finite subgroup of $O(V)$ preserving $R$ is the \emph{automorphism group} of $R$ and is denoted by $A(R)$. It contains $W(R)$.

A \emph{basis} for $R$ is a family $\{\alpha_1,\dots,\alpha_l\}\subset R$ such that, for every $\alpha\in R$, either $\alpha$ or $-\alpha$ is a combination of the $\alpha_m$ with nonnegative integer coefficients. One can for instance take 
$$\alpha_1=\varepsilon_1-\varepsilon_2,\alpha_2=\varepsilon_2-\varepsilon_3,\dots,\alpha_{l-1}=\varepsilon_{l-1}-\varepsilon_l, \alpha_l = \varepsilon_{l-1}+\varepsilon_l.$$

When $l>4$, the quotient $A(R)/W(R)$ is isomorphic to $\mathbb Z/2$. When $l=4$, with $\alpha_1,\dots,\alpha_4$ as above, the quotient $A(R)/W(R)$ is isomorphic to the permutation group $S_3$ of $\{1,3,4\}$ in the following way: for every
$a \in A(R)$, there exists a unique $w \in W(R)$ and a unique permutation $\tau$ of $\{1,3,4\}$ such that
$a(\alpha_2) = w(\alpha_2)$ and $a(\alpha_m)=w(\alpha_{\tau(m)})$ for $m=1,3,4$.

\medskip

In our case, we will have $l=4$ and $V=H_1^{(0)}$ with the scalar product making $\{\varepsilon_1,\varepsilon_i,\varepsilon_j,\varepsilon_k\}$ an orthonormal basis. We take $R = \{\varepsilon_g + \varepsilon_h \,; \; g,h \in Q, \,g \ne \pm h\,\} = \{\pm\varepsilon_g \pm \varepsilon_h \,;\; g,h \in \{1,i,j,k\}, g\ne h \,\}$; they form a root system of $D_4$-type. Our preferred basis of $R$ is $\alpha_1=\varepsilon_1-\varepsilon_i$, $\alpha_2=\varepsilon_i-\varepsilon_j$, $\alpha_3=\varepsilon_j-\varepsilon_k$ and $\alpha_4=\varepsilon_j+\varepsilon_k$.

From the description of the Weyl group and \eqref {e.aut-g3} it follows that $Q$ acts on $H_1^{(0)}$ by some explicit elements of the Weyl group $W(R)$.

From \eqref{e.S5-g3}, \eqref{e.T5-g3}, $\widetilde{S}$, $\widetilde{T}$ act on $H_0$ by  elements of $A(R)$:\\

\begin{center}
{\renewcommand{\arraystretch}{1.5}
\renewcommand{\tabcolsep}{0.2cm}
\begin{tabular}{|c|c|c|}
\hline
 & $\widetilde{S}$ & $\widetilde{T}$  \\
\hline
$\alpha_1 = \varepsilon_1-\varepsilon_i$ & $\varepsilon_j+\varepsilon_k$ & $\varepsilon_1-\varepsilon_k$ \\
\hline
$\alpha_2 = \varepsilon_i-\varepsilon_{j}$ & $\varepsilon_1-\varepsilon_j$ & $\varepsilon_i-\varepsilon_1$ \\
\hline
$\alpha_3=\varepsilon_j-\varepsilon_{k}$ & $\varepsilon_j-\varepsilon_k$ & $\varepsilon_j-\varepsilon_i$ \\
\hline
$\alpha_4=\varepsilon_j+\varepsilon_{k}$ & $\varepsilon_i-\varepsilon_1$ & $\varepsilon_1+\varepsilon_k$ \\
\hline
\end{tabular}}
\end{center}

\bigskip

The action of $\textrm{Aff}(M_3,\omega_{(3)})$ on $H_1^{(0)}$ is therefore given by a homomorphism
$$Z:\textrm{Aff}(M_3,\omega_{(3)})\to A(R)$$
for which we obtain in the sequel an explicit description. Let $p$ be the canonical projection $p:A(R)\to A(R)/W(R)\simeq S_3$; consider
$$\pi:\textrm{Aff}(M_3,\omega_{(3)})\to A(R)/W(R)$$
given by $\pi=p\circ Z$. 

\begin{lemma} The homomorphism $\pi$ is surjective. Its kernel is the inverse image $\Gamma^*(2)$ of $\Gamma(2)$ under the natural homomorphism from $\textrm{Aff}(M_3,\omega_{(3)})$ to $SL(2,\mathbb Z)$. 
\end{lemma}

\begin{proof}
Define $w_S,w_T \in W(R)$ by
$$w_S(\varepsilon_1)=\varepsilon_i,\quad w_S(\varepsilon_i)=\varepsilon_1,\quad w_S(\varepsilon_j)=\varepsilon_j,\quad w_S(\varepsilon_k)=\varepsilon_k,$$
$$w_T(\varepsilon_1)=\varepsilon_j,\quad
w_T(\varepsilon_j)=\varepsilon_1,\quad w_T(\varepsilon_i)=\varepsilon_i,\quad w_T(\varepsilon_k)=\varepsilon_k.$$
We have then
$$w_S \circ \widetilde S (\alpha_1) = \alpha_4,\quad  w_S \circ \widetilde S (\alpha_2) = \alpha_2,\quad  w_S \circ \widetilde S (\alpha_3) = \alpha_3,\quad  w_S \circ \widetilde S (\alpha_4) = \alpha_1,$$
$$w_T \circ \widetilde T (\alpha_1) = \alpha_3,\quad  w_T \circ \widetilde T (\alpha_2) = \alpha_2,\quad  w_T \circ \widetilde T (\alpha_3) = \alpha_1,\quad  w_T \circ \widetilde T (\alpha_4) = \alpha_4.$$

This shows that the images of $\widetilde S ,\widetilde T$ by $\pi$ are the transpositions $(1,4)$, $(1,3)$ respectively. Therefore $\pi$ is onto.

\medskip

 Denote by $\widetilde{e^{i\pi}}$ the element $(\widetilde{S}\widetilde{T}^{-1}\widetilde{S})^2$ of $\textrm{Aff}_{(1)}(M_3,\omega_{(3)})$; its  image in $SL(2,\mathbb Z)$ is $-\textrm{Id}\in \Gamma(2)$. The actions of $\widetilde{S}^2$, $\widetilde{T}^2$ and $\widetilde{e^{i\pi}}$ on $H_1^{(0)}$ are given by
\begin{center}
{\renewcommand{\arraystretch}{1.5}
\renewcommand{\tabcolsep}{0.2cm}
\begin{tabular}{|c|c|c|c|}
\hline
 & $\widetilde{S}^2$ & $\widetilde{T}^2$ & $\widetilde{e^{i\pi}}$  \\
\hline
$\varepsilon_1$ & $\varepsilon_i$ & $\varepsilon_j$ & $\varepsilon_k$ \\
\hline
$\varepsilon_i$ & $\varepsilon_1$ & $\varepsilon_k$ & $-\varepsilon_j$\\
\hline
$\varepsilon_j$ & $-\varepsilon_k$ & $\varepsilon_1$ & $\varepsilon_i$ \\
\hline
$\varepsilon_{k}$ & $-\varepsilon_j$ & $\varepsilon_i$ & $-\varepsilon_1$ \\
\hline
\end{tabular}}
\end{center}

\medskip

Therefore $\widetilde{S}^2$, $\widetilde{T}^2$ and $\widetilde{e^{i\pi}}$ belong to $\textrm{Ker}(\pi)$. But it is well-known that $\Gamma(2)$ is generated by $S^2,T^2$ and $-\textrm{Id}$. As the kernel of $\pi$ contains $Q$, it contains
$\Gamma^*(2)$. Finally, as the quotient $\textrm{Aff}(M_3,\omega_{(3)})/ \Gamma^*(2) \simeq SL(2, \mathbb Z) / \Gamma(2)  $
is isomorphic to $S_3$, the kernel of $\pi$ is equal to $\Gamma^*(2)$.
\end{proof}

\begin{remark}The element $\widetilde{e^{i\pi}}$ satisfies $(\widetilde{e^{i\pi}})^2=-1\in Q$. It is easily checked that it commutes with $\widetilde{S}$, $\widetilde{T}$ and the elements of $Q$. Therefore, it belongs to the center of $\textrm{Aff}(M_3,\omega_{(3)})$. The center of $\textrm{Aff}(M_3,\omega_{(3)})$ is in fact equal to the cyclic subgroup of order $4$ generated by $\widetilde{e^{i\pi}}$. Indeed, the center of $\textrm{Aff}(M_3,\omega_{(3)})$ has to be contained in the inverse image in $\textrm{Aff}(M_3,\omega_{(3)})$ of the center $\{\pm \textrm{Id}\}$ of $SL(2,\mathbb{Z})$. This inverse image is the subgroup of order 16 (isomorphic to the subgroup $\widetilde Q$ below) whose center is generated by $\widetilde{e^{i\pi}}$ (see also~\cite{HS}).
\end{remark}

\bigskip

We finally study the restriction
$$Z: \Gamma^*(2) \rightarrow W(R).$$

Let
$$\Gamma(4):=\left\{M\equiv \textrm{Id}_{SL(2,\mathbb{Z})} \, \textrm{mod}\, 4\right\}.$$

Then $\Gamma(2)/\Gamma(4)$ is isomorphic to $(\mathbb Z /2)^3$, generated by the images of $S^2,T^2$ and $-\textrm{Id}$.
Let $s \in W(R)$ be the element defined by
$$s(\varepsilon_1) = \varepsilon_1,\quad s(\varepsilon_i)=-\varepsilon_i, \quad s(\varepsilon_j)=-\varepsilon_j,\quad s(\varepsilon_k)=\varepsilon_k.$$
Identify $Q$ with its image in $W(R)$.

\begin{lemma}\begin{enumerate}
\item One has $s^2=1,\;is=-si,\;js=-sj,\;ks=sk$.
\item The subgroup $\widetilde Q$ of $W(R)$ generated by $Q$ and $s$ has order $16$ and is formed of the elements $\pm 1,\pm i,\pm j,\pm k,\pm s,\pm is,\pm js,\pm ks$.
\item There is a unique morphism from $\widetilde Q$ to $(\mathbb Z /2)^3 \simeq \Gamma(2)/\Gamma(4)$ sending  $is, js, ks$ to the images of $S^2, T^2, -\textrm{Id}$  in $\Gamma(2)/\Gamma(4)$ respectively; the sequence 
$$1 \rightarrow \{\pm 1\} \rightarrow \widetilde Q \rightarrow (\mathbb Z /2)^3 \simeq \Gamma(2)/\Gamma(4) \rightarrow 1$$ 
is exact.
\end{enumerate}
\end{lemma}
\begin{proof}
The first part is a direct calculation; it shows that the elements 
$$\pm 1,\pm i,\pm j,\pm k,\pm s,\pm is,\pm js,\pm ks$$ 
form a subgroup of $W(R)$; as $s\notin Q$, these elements must be distinct. This proves the second part of the lemma. The subgroup $\{\pm1\}$ is contained in the center of $\widetilde Q$; and the quotient $\widetilde Q/\{\pm1\}$ is abelian with every element of order $\leq 2$, hence isomorphic to $(\mathbb Z /2)^3$. Moreover, the image of $is,js,ks$ clearly generate this quotient. This proves the third part of the lemma.
\end{proof}

The formulas for the actions of $\widetilde{S}^2$, $\widetilde{T}^2$ and $\widetilde{e^{i\pi}}$ on the $\varepsilon_g$, $g\in Q$, show that the images of these elements in $W(R)$ are respectively equal to $is, js, ks$.

\begin{lemma} The image $Z(\Gamma^*(2))$ is equal to $\widetilde Q$, which is also equal to the image $Z(\widetilde \Gamma(2))$. The 
inverse image  $Z^{-1}(\{\pm 1\}) \bigcap \widetilde \Gamma(2)$ is equal to the inverse image $\widetilde \Gamma(4)$ of $\Gamma(4)$
in  \linebreak $\textrm{Aff}_{(1)}(M_3,\omega_{(3)})$. The restriction of the canonical map $\widetilde \Gamma(2) \rightarrow \Gamma(2)$ to the kernel of $Z$ in $\widetilde\Gamma(2)$ is an isomorphism onto $\Gamma(4)$.
\end{lemma}

\begin{proof}
The first assertion follows from the fact that the images by $Z$ of the generators $\widetilde{S}^2$, $\widetilde{T}^2$ and $\widetilde{e^{i\pi}}$ of $\widetilde \Gamma(2)$ are the generators $is,js,ks$ of $\widetilde Q$. The second assertion follows from the exact sequence just before the statement of the lemma. The third assertion follows from the fact that the kernel of the morphism $\widetilde \Gamma(4) \rightarrow \Gamma(4)$ is sent isomorphically by $Z$ onto $\{\pm 1\}$.
\end{proof}

\begin{remark}
It is not difficult to determine explicitly the kernel of $Z$ in $\widetilde \Gamma(2)$. One first observes, looking at the action of
$\Gamma(4)$ in the upper half plane, that $\Gamma(4)$ is the free subgroup generated by

$$S^4 = \left( \begin{array}{cc} 1 & 0 \\  4 & 1 \\  \end{array} \right),\quad T^4 = \left( \begin{array}{cc} 1 & 4 \\  0 & 1 \\  \end{array} \right),\quad (TS)^3= \left( \begin{array}{cc} 13 & 8 \\  8 & 5 \\  \end{array} \right),$$
$$ -(S^2T)^2= -\left( \begin{array}{cc} 3 & 4 \\  8 & 11 \\  \end{array} \right),\quad S^2 T^4 S^2 = \left( \begin{array}{cc} 9 & 4 \\  20 & 9 \\  \end{array} \right).$$

Then, from the action of $\widetilde S$, $\widetilde T$ on the $\alpha_i$, $1\leq i \leq 4$, one checks that the corresponding generators for the kernel of $Z$ in $\widetilde \Gamma(2)$ are

$$\widetilde S^4,\; \widetilde T^4,\; (-1).(\widetilde T \widetilde S)^3,\; \widetilde{e^{i\pi}} (\widetilde S^2 \widetilde T)^2,\;
\widetilde S^2 \widetilde T^4 \widetilde S^2.$$

\end{remark}

\begin{lemma}
The subgroup $\widetilde{Q}\subset W(R)$ is \emph{exactly} the subgroup of symplectic elements of $W(R)$ w.r.t. the symplectic form induced by the intersection form on $H_1^{(0)}$.

\end{lemma}

\begin{proof}
 A direct inspection reveals that the intersection form on the elements $\widehat{\sigma}_g$, $g\in\{1,i,j,k\}$, is given by 
\begin{center}
{\renewcommand{\arraystretch}{1.5}
\renewcommand{\tabcolsep}{0.2cm}
\begin{tabular}{|c|c|c|c|c|}
\hline
$(.,.)$ & $\widehat{\sigma}_1$ & $\widehat{\sigma}_i$ & $\widehat{\sigma}_j$ & $\widehat{\sigma}_k$ \\
\hline
$\widehat{\sigma}_1$ & $0$ & $+2$ & $0$ & $0$ \\
\hline
$\widehat{\sigma}_i$ & $-2$ & $0$ & $0$ & $0$\\
\hline
$\widehat{\sigma}_j$ & $0$ & $0$ & $0$ & $-2$ \\
\hline
$\widehat{\sigma}_k$ & $0$ & $0$ & $+2$ & $0$ \\
\hline
\end{tabular}}
\end{center}
\medskip
Using the definition $\varepsilon_g:=\widehat{\sigma}_g-\widehat{\sigma}_{gj}$, we obtain the following intersection table:
\begin{center}
{\renewcommand{\arraystretch}{1.5}
\renewcommand{\tabcolsep}{0.2cm}
\begin{tabular}{|c|c|c|c|c|}
\hline
$(.,.)$ & $\varepsilon_1$ & $\varepsilon_i$ & $\varepsilon_j$ & $\varepsilon_k$ \\
\hline
$\varepsilon_1$ & $0$ & $0$ & $0$ & $+4$ \\
\hline
$\varepsilon_i$ & $0$ & $0$ & $-4$ & $0$\\
\hline
$\varepsilon_j$ & $0$ & $+4$ & $0$ & $0$ \\
\hline
$\varepsilon_k$ & $-4$ & $0$ & $0$ & $0$ \\
\hline
\end{tabular}}
\end{center}
\medskip
In particular, since any symplectic transformation sends symplectic planes into symplectic planes, we conclude that symplectic elements of $W(R)$ either fix or exchange the subsets $\{\pm\varepsilon_1,\pm\varepsilon_k\}$ and $\{\pm\varepsilon_j,\pm\varepsilon_i\}$. It follows that the subgroup of symplectic elements of $W(R)$ has order 16, so it is equal to $\widetilde{Q}$.
\end{proof}

The proof of Theorem~\ref{t.A} is now complete.


\section{Proof of Theorem~\ref{t.B}}

The outline of this section is the following. In Subsection 3.1, we introduce a family of square-tiled surfaces parametrized by an odd integer $q\geq 3$ such that the one featuring in Theorem~\ref{t.B} corresponds to $q=3$. The Veech group is computed, as well as a representation of these surfaces as algebraic curves. We then compute in Subsection 3.2 the action of the automorphism group (cyclic of order $q$) on homology, and then in Subsection 3.3 the action of generators of the affine group. This allows in Subsection 3.4 to determine an invariant complement $H_{rel}$ of $H_1(M,\mathbb{Q})$ in $H_1(M,\Sigma,\mathbb{Q})$. The subspace $H_1^{(0)}$ breaks into invariant subspaces $H_{\tau}$ and $\breve{H}$ which are analyzed in Subsections 3.5, 3.6 respectively (for any odd $q\geq 3$). The special case $q=3$ (corresponding to a larger Veech group equal to $SL(2,\mathbb{Z})$) is further analyzed in Subsection 3.7.

\subsection{The square-tiled surface $(M,\omega)$}

Let $q$ be an {\bf odd} integer (for the surface referred to in Theorem~\ref{t.B}, $q$ is equal to $3$). Let $(M,\omega)$ be the square-tiled surface such that
\begin{itemize}
\item The set $Sq(M,\omega)$ is identified with $\mathbb Z/q \times \mathbb Z/2 \times \mathbb Z/2 $; we denote as before by $sq(g)$ the square associated to $g$;
\item The map $r$ (for right) is 
$$ \begin{array}{ll}
& sq(i,\mu,\nu) \mapsto \left\{ \begin{array}{l} sq(i,\mu +1, \nu) \quad \hbox{for} \quad \mu =1,\\[2ex]
sq(i+1,\mu +1, \nu) \quad \hbox{for} \quad \mu =0,\;\nu =0, \\[2ex]
sq(i-1,\mu +1, \nu) \quad \hbox{for} \quad \mu =0,\;\nu =1. \end{array} \right. 
\end{array}$$
\item The map $u$ (for up) is 
$$ \begin{array}{ll}
& sq(i,\mu,\nu) \mapsto \left\{ \begin{array}{l} sq(i,\mu , \nu +1) \quad \hbox{for} \quad \nu =1,\\[2ex]
sq(i+1,\mu, \nu +1) \quad \hbox{for} \quad \nu =0,\;\mu =1, \\[2ex]
sq(i-1,\mu, \nu +1) \quad \hbox{for} \quad \nu =0,\;\mu =0. \end{array} \right. 
\end{array}$$
\end{itemize}
\begin{figure}[!h]
\centering
\includegraphics[scale=0.35]{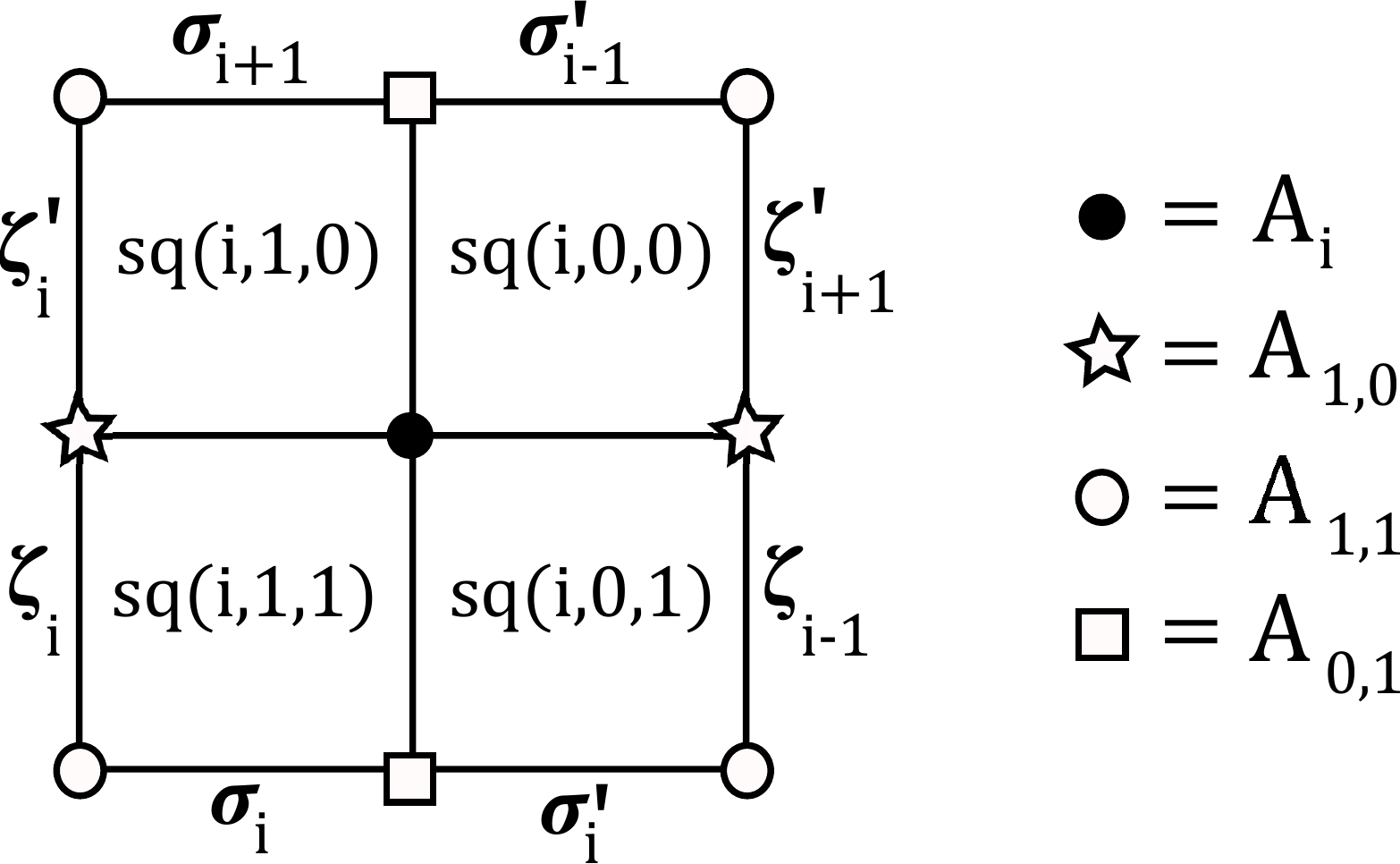}
\caption{Description of $sq(i,\mu,\nu)$, $i\in\Zq$, $\mu,\nu\in\mathbb{Z}/2$.}\label{coordinates}
\end{figure}

The commutator $uru^{-1}r^{-1}$ maps $sq(i,\mu,\nu)$ to $sq(i,\mu,\nu)$ if $(\mu, \nu) = (0,0)$, \linebreak to $sq(i-2,\mu,\nu)$ if $(\mu, \nu) = (1,0)$ or $(0,1)$, to $sq(i+4,\mu,\nu)$ if $(\mu, \nu) = (1,1)$. As $q$ is odd, this shows that, for $(\mu,\nu) \ne (0,0)$,
 the lower left corners of the $q$ squares $sq(i,\mu,\nu)$, $i \in \mathbb Z /q$, coincide and define a point $A_{\mu,\nu}$ of $\Sigma$ of ramification index $q$, i.e., a zero of $\omega$ of order $q-1$. On the other hand, the lower left corners $A_{0,0,i}=:A_i$ of the squares
 $sq(i,0,0)$ are regular points of $(M,\omega)$. We have thus $\Sigma = \{A_{0,1},A_{1,0},A_{1,1} \}$. The form $\omega$ has a zero of even order
 $q-1$ at each point of $\Sigma$ and the genus $g$ of $M$ is $\frac{3q-1}{2}$. Compare with Figure~\ref{coordinates}.
 
In Figures~\ref{q3} and~\ref{q5} below, we illustrate the square-tiled surfaces corresponding to the cases $q=3$ and $q=5$.

The automorphism group ${\rm Aut} (M,\omega)$ is $\mathbb Z /q$, acting by translations of the $i$ variable. It fixes each point of $\Sigma$.\\
 
Let us denote by ${\rm Aff}_{(1)} (M,\omega) $ the stabilizer of $A_0$ in the affine group ${\rm Aff} (M,\omega) $. Taking the differential
defines an isomorphism from ${\rm Aff}_{(1)} (M,\omega) $ onto the Veech group $SL(M,\omega)$.\\

When $q$ is odd $\geq 5$, $i-2$ is not equal to $i+4$ in $\mathbb Z /q$ and therefore every element of  ${\rm Aff} (M,\omega) $ must fix $A_{1,1}$.
On the other hand, one sees easily that the matrices 
$$S^2 = \left( \begin{array}{cc} 1 & 0 \\  2 & 1 \\  \end{array} \right),\quad T^2 = \left( \begin{array}{cc} 1 & 2 \\  0 & 1 \\  \end{array} \right),\quad J = \left( \begin{array}{cc} 0 & -1 \\  1 & 0 \\  \end{array} \right),$$
belong to $SL(M,\omega)$. We have thus (as $\Gamma(2)$ is generated by $S^2,T^2,J^2$)
$$ SL(M,\omega) = \{ M \in SL(2,\mathbb Z); \; M \equiv  {\rm Id} \;\;{\rm or}\;\; M \equiv J \;\;{\rm mod}\; 2 \;\}\;.$$

\begin{figure}[h!]
\centering
\includegraphics[scale=0.28]{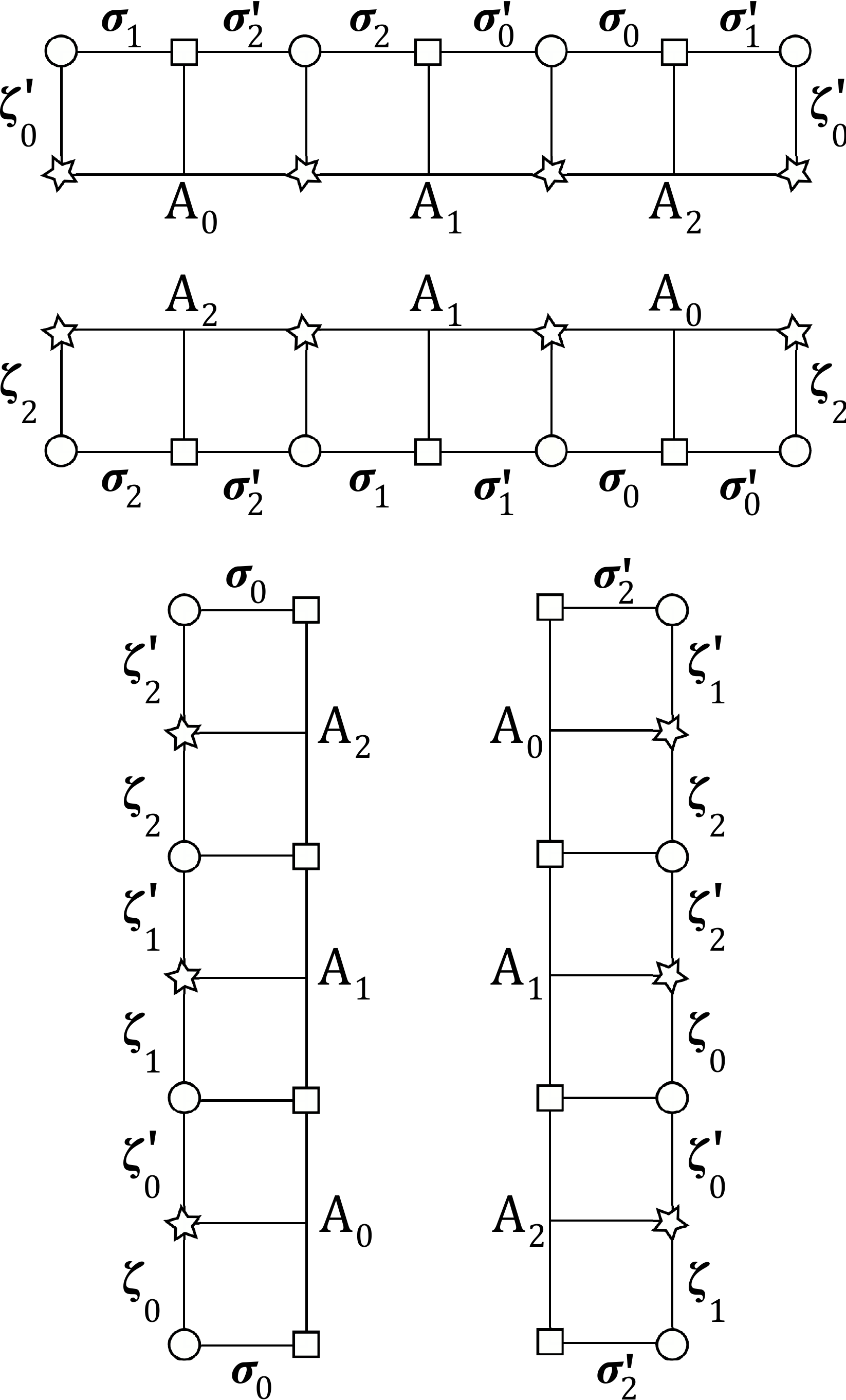}
\caption{Horizontal and vertical cylinder decompositions when $q=3$.}\label{q3}
\end{figure}
\newpage
\begin{figure}[!h]
\centering
\includegraphics[scale=0.35]{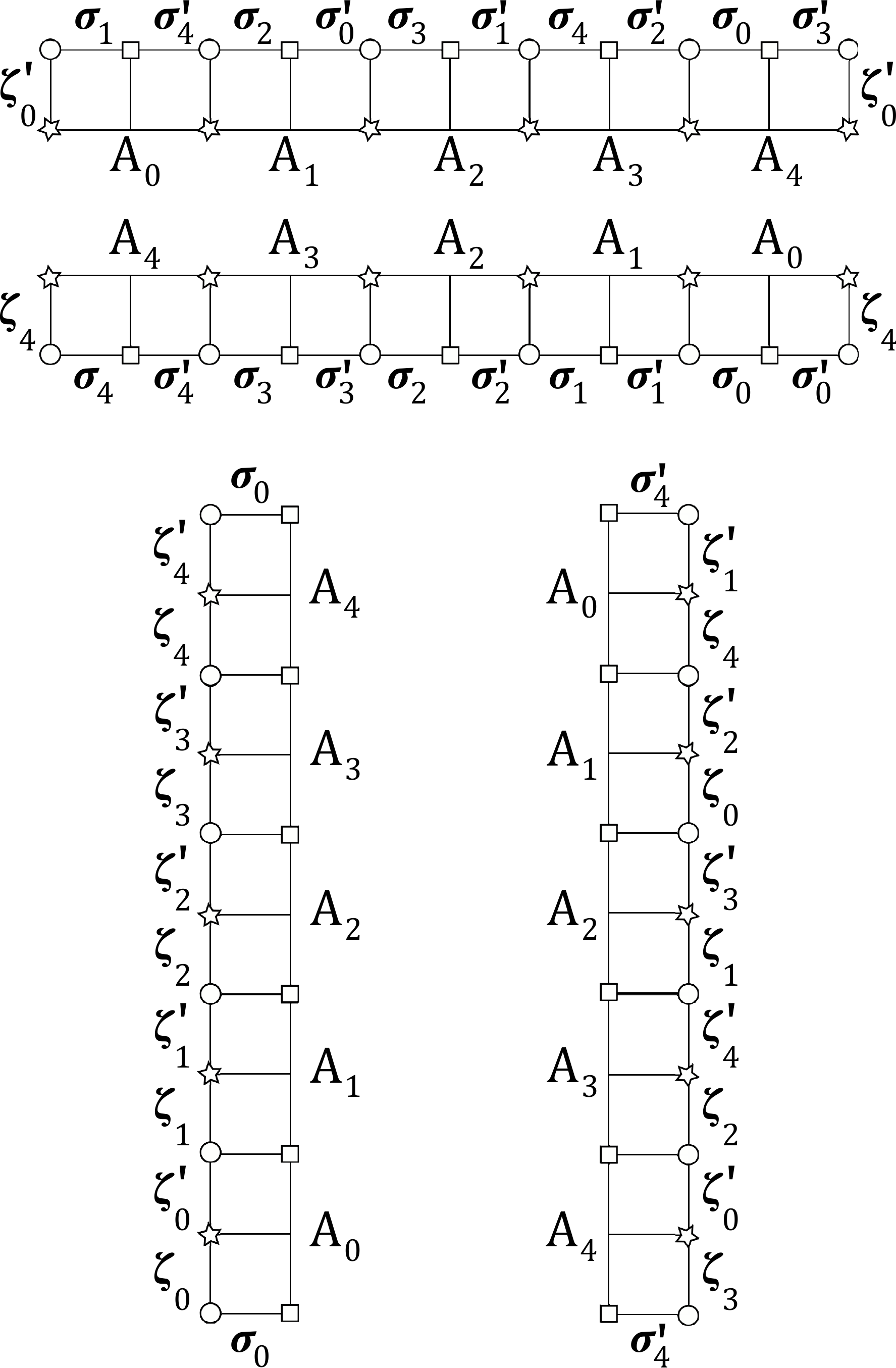}
\caption{Horizontal and vertical cylinder decompositions when $q=5$.}\label{q5}
\end{figure}

When $q=3$, it is easy to see that $S,\,T \in SL(M,\omega)$. The Veech group in this case is equal to $SL(2,\mathbb Z)$.\\

In all cases, by considering generators of $SL(M,\omega)$, one checks that any element of ${\rm Aff}_{(1)} (M,\omega)$ actually fixes all the $A_i$.
It follows that ${\rm Aut} (M,\omega)$ is contained in the center of  ${\rm Aff} (M,\omega) $, and that ${\rm Aff} (M,\omega) $ is the direct product 
of  ${\rm Aut} (M,\omega) \simeq \mathbb Z/q$ by $SL(M,\omega)$. \\

\begin{remark}Let $t \in {\rm Aff} (M,\omega)$ be the element whose linear part is $-\textrm{Id}\in SL(2,\mathbb{Z})$ and such that $t(sq(i,\mu,\nu))
= sq(i+1,\mu +1,\nu +1)$ for all $i,\mu,\nu$. It has order $2q$, fixes each point in $\Sigma$ and permutes cyclically the $A_{0,0,i}$.
The quotient $(M,\omega) / \langle t\rangle$ is by the Riemann-Hurwitz formula of genus $0$. 

Although we do not need it in the following, let us give formulas associated to this cyclic covering.

We claim that $M$ is isomorphic to the desingularization of the algebraic curve 
$$w^{2q}= z^{q-2}(z^2-1),$$
with the holomorphic form 
$$\omega^{*}=c^{-1}\frac{z^{\frac{q-3}{2}}dz}{w^q},$$
and the automorphism $t$ given by 
$$(w,z)\mapsto \left(\exp\left(-\frac{\pi i}{q}\right)w, z\right).$$
Here, $c$ is the elliptic integral 
$$c=\int_1^\infty \frac{dx}{\sqrt{x^3-x}} = \int_0^1\frac{dx}{\sqrt{x-x^3}} = \frac{(\Gamma(1/4))^2}{2\sqrt{2\pi}}$$
(cf.~\cite{BB}, thm. 1.7).

Indeed, let $M^*$ be the curve $\{w^{2q}= z^{q-2}(z^2-1)\}$, desingularized in order to have $q$ distinct points at infinity given by $z=\exp\left(\frac{2\pi ij}{q}\right)w^2+O(1)$, $j\in\Zq$. Observe that $(c\omega^*)^2$ is the pull-back by the projection $(z,w)\mapsto z$ of the quadratic differential with simple poles $\frac{(dz)^2}{z(z^2-1)}$. The elliptic curve $E=\{y^2=z(z^2-1)\}$ is isomorphic to $\mathbb{C}/2\mathbb{Z}+2i\mathbb{Z}$. More precisely, let $E'=\{(y,z)\in E: z\notin [-1,1]\}$; the formula 
$$P\mapsto \int_\infty^P c^{-1}\frac{dz}{y}$$
(with the path of integration contained in $E'$) provides a biholomorphism $\Phi$ from $E'$ onto $(-1,1)^2$, such that $\Phi^{-1}$ extends continuously from $[-1,1]^2$ to $E$ with 
$$\Phi^{-1}(\pm1,0) = (0,1),$$
$$\Phi^{-1}(0,\pm1)=(0,-1),$$
$$\Phi^{-1}(\pm1,\pm1)=(0,0).$$
Similarly, let 
$$M'=\{(w,z)\in M^*: z\notin[-1,1]\}=\{(w,z)\in M^*: \left(1-\frac{1}{z^2}\right)\notin\{0\}\cup\mathbb{R}^-\}.$$ 
Then, $M'$ has $q$ connected components indexed naturally by $\Zq$, with 
$$M_j'=\left\{(w,z)\in M': \left|\arg\frac{z}{w^2} - \frac{2\pi j}{q}\right|<\frac{\pi}{q}\right\}.$$
Each connected component $M_j'$ contains one point at infinity; the formula 
$$P\mapsto \int_\infty^P \omega^*$$ 
(with the path of integration contained in $M_j'$) provides a biholomorphism $\Phi_j$ from $M_j'$ onto $(-1,1)^2$, such that $\Phi_j^{-1}$ extends continuously from $[-1,1]^2$ to $M^*$ with 
$$\Phi_j^{-1}(\pm1,0) = (0,1),$$
$$\Phi_j^{-1}(0,\pm1)=(0,-1),$$
$$\Phi_j^{-1}(\pm1,\pm1)=(0,0).$$
One also sees that, for $0\leq s\leq 1$, $j\in\Zq$, 
$$\Phi_j^{-1}(1,s) = \Phi_{j+1}^{-1}(-1,s) ,$$
$$\Phi_j^{-1}(1,-s)=\Phi_{j-1}^{-1}(-1,-s),$$
$$\Phi_j^{-1}(s,1)=\Phi_{j-1}^{-1}(s,-1),$$
$$\Phi_j^{-1}(-s,1)=\Phi_{j+1}^{-1}(-s,-1).$$
This completes the justification of our claim that $(M^*,\omega^*)$ is isomorphic to $(M,\omega)$. The $q$ points at infinity on $M^*$ correspond to the $A_{0,0,j}$, the points $(0,1)$, $(0,-1)$, $(0,0)$ of $M^*$ correspond to the points $A_{0,1}, A_{1,0}, A_{1,1}$ of $\Sigma\subset M$. The sides of the squares correspond to the locus on $M^*$ where the function $z$ is real. See Figure~\ref{f.remark} below. The claim about $t$ is easily checked.

When $q=3$, one recovers the equation of Subsection 1.4 (except that $z=0$ has been sent to infinity).
\begin{figure}[!h]
\centering
\includegraphics[scale=0.35]{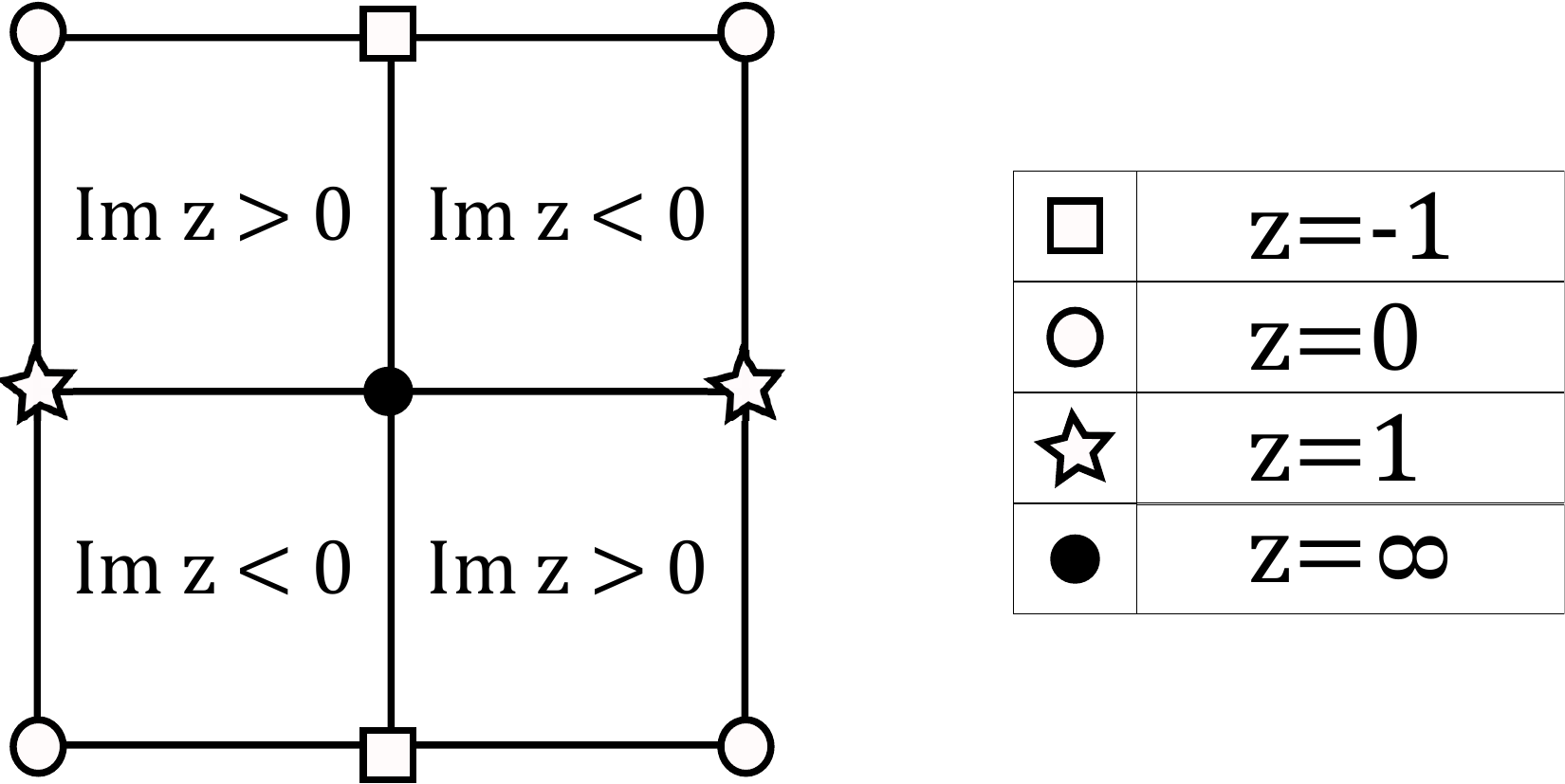}
\caption{}\label{f.remark}
\end{figure}

\end{remark}

\subsection{Action of  ${\rm Aut} (M,\omega)$ on $H_1(M,\Sigma,\mathbb{Z})$}

We consider the direct sum of four copies of ${\mathbb Z}(\Zq)$; the canonical bases of the four copies are respectively denoted by 
$(\sigma_i)_{i \in \Zq}$, $(\sigma'_i)_{i \in \Zq}$, $(\zeta_i)_{i \in \Zq}$, $(\zeta'_i)_{i \in \Zq}$. We define a homomorphism $p$ from this 
direct sum onto $H_1(M,\Sigma,\mathbb{Z})$ by sending
\begin{itemize}
\item $\sigma_i$ on the bottom side of $sq(i,1,1)$, oriented rightwards;
\item $\sigma'_i$ on the bottom side of $sq(i,0,1)$, oriented rightwards;
\item $\zeta_i$ on the left side of $sq(i,1,1)$, oriented upwards;
\item $\zeta'_i$ on the left side of $sq(i,1,0)$, oriented upwards.
\end{itemize}

See Figure~\ref{coordinates} above.

The homomorphism $p$ is compatible with the actions of $\Zq$ on $\mathbb{Z}(\Zq) \oplus \mathbb{Z}(\Zq) \oplus \mathbb{Z}(\Zq) \oplus \mathbb{Z}(\Zq)$ (by the regular representation)
 and $H_1(M,\Sigma,\mathbb{Z})$ (identifying $\Zq$ with ${\rm Aut} (M,\omega)$). The kernel of the homomorphism $p$ is the submodule $Ann$ of $\mathbb{Z}(\Zq) \oplus \mathbb{Z}(\Zq)\oplus \mathbb{Z}(\Zq)\oplus \mathbb{Z}(\Zq)$ generated by the elements
\begin{equation}\label{e.relations-g4}
\Box_i :=\sigma_i+ \sigma'_i +\zeta_{i-1} + \zeta'_{i+1} -\sigma'_{i-1}-\sigma_{i+1} -\zeta'_i - \zeta_i .
\end{equation}

We have $\sum_i \Box_i =0$, hence $Ann$ has rank $q-1$.  The boundary map $\partial: H_1(M,\Sigma,\mathbb{Z}) \rightarrow
\mathbb{Z}(\Sigma)$ is induced by
\begin{eqnarray*}
\partial (p(\sigma_i)) &=& A_{0,1} - A_{1,1} , \\
\partial (p(\sigma'_i)) &=& A_{1,1} - A_{0,1} , \\
\partial (p(\zeta_i)) &=& A_{1,0} - A_{1,1} , \\
\partial (p(\zeta'_i)) &=& A_{1,1} - A_{1,0} .
\end{eqnarray*}
\medskip

We introduce
\begin{eqnarray*}
\sigma &:= & p(\sum_i \sigma_i + \sum_i \sigma'_i), \\  
\zeta &:= & p(\sum_i \zeta_i + \sum_i \zeta'_i),\\
\sigma^{\flat} &:= & p(\sum_i \sigma_i - \sum_i \sigma'_i), \\
\zeta^{\flat} &:= & p(\sum_i \zeta_i - \sum_i \zeta'_i).
\end{eqnarray*}
 
We have 
$$g.\sigma =\sigma, \quad  g.\zeta = \zeta,\quad g.\sigma^{\flat} =\sigma^{\flat}, \quad g.\zeta^{\flat}=\zeta^{\flat},$$
for all $g \in \Zq \simeq {\rm Aut} (M,\omega)$. We have also
$$\partial \sigma = \partial \zeta =0, \quad \partial \sigma^{\flat} = 6(A_{0,1}-A_{1,1}), \quad \partial \zeta^{\flat} = 6(A_{1,0}-A_{1,1}). $$
 
Let
$$a_i := p(\sigma_i -\sigma_{i+1}),\quad a'_i := p(\sigma'_i -\sigma'_{i+1}), \quad b_i := p(\zeta_i -\zeta_{i+1}), \quad b'_i := p(\zeta'_i -\zeta'_{i+1}).$$

We have

$$\sum_i a_i = \sum_i a'_i =\sum_i b_i =\sum_i b'_i =0,$$

$$\partial a_i = \partial a'_i =\partial b_i=\partial b'_i=0,$$

and

$$g.a_i = a_{i+g},\quad g.a'_i = a'_{i+g},\quad g.b_i = b_{i+g},\quad g.b'_i = b'_{i+g},$$

for $i,g \in \Zq$. From the formula for $\Box_i$, we have, for $i \in \Zq$, 

$$a_i -a'_{i-1}+b_{i-1}-b'_i=0.$$

We define 
$$ \tau_i := a_i - a'_{i-1} = b'_i - b_{i-1},\quad \breve{\sigma}_i := a_i + a'_{i-1}, \quad \breve{\zeta}_i := b'_i + b_{i-1}.$$

We have

$$\sum_i \tau_i = \sum_i \breve{\sigma}_i = \sum_i \breve{\zeta}_i = 0$$
and

$$g.\tau_i = \tau_{i+g},\quad g.\breve{\sigma}_i = \breve{\sigma}_{i+g}, \quad g.\breve{\zeta}_i = \breve{\zeta}_{i+g}.$$

We observe that $\{\tau_i, \breve{\sigma_i}, \breve{\zeta_i};\; i\in \Zq,\; i\ne 0\}$ form a basis for $H_{1}^{(0)}(M, \mathbb Q)$. It will also be clear
in the next subsection that $\sigma, \zeta$ form a basis for $H_{1}^{st}(M, \mathbb Q)$, and that $\sigma^{\flat}, \zeta^{\flat}$ form a basis for a subspace of $H_1(M,\Sigma, \mathbb Q)$ complementing  $H_1(M, \mathbb Q)$ and invariant under the affine group ${\rm Aff} (M, \omega)$.

\subsection{Action of  ${\rm Aff} (M,\omega)$ on $H_1(M,\Sigma,\mathbb{Z})$} 

When $q=3$, we denote by $\wt S$, $\wt T$ the elements of ${\rm Aff}_{(1)} (M,\omega)$ with linear part $S,T$ respectively. When $q\geq 5$, we denote
by $ \wt S^2, \wt T^2, \wt J$ the elements of ${\rm Aff}_{(1)} (M,\omega)$ with linear part $S^2,T^2,J$ respectively (although $\wt S, \wt T$ are not defined in this case).\\

When $q=3$, one has

\medskip

\begin{center}
{\renewcommand{\arraystretch}{1.5}
\renewcommand{\tabcolsep}{0.2cm}
\begin{tabular}{|c|c|c|c|c|}
\hline
 & $p(\sigma_i)$ & $p(\sigma'_i)$ & $p(\zeta_i)$ & $p(\zeta'_i)$ \\
\hline
$\wt S$&$p(\sigma_i)+ p(\zeta'_{i-1})$&$p(\sigma'_i)+ p(\zeta_{i-1})$&$p(\zeta'_{i-1})$& $p(\zeta_i)$  \\
\hline
$\wt T$&$p(\sigma'_{i+1})$&$p(\sigma_i)$&$p(\zeta_{i})+p(\sigma'_{i+1})$& $p(\zeta'_i)+p(\sigma_{i+1})$  \\
\hline
\end{tabular}}
\end{center}

\medskip

\begin{center}
{\renewcommand{\arraystretch}{1.5}
\renewcommand{\tabcolsep}{0.2cm}
\begin{tabular}{|c|c|c|c|c|}
\hline
 &$\sigma$&$\zeta$&$\sigma^{\flat}$&$\zeta^{\flat}$ \\
\hline
$\wt S$&$\sigma +\zeta$&$\zeta$&$\sigma^{\flat}-\zeta^{\flat}$&$-\zeta^{\flat}$  \\
\hline
$\wt T$&$\sigma $&$\zeta +\sigma$&$-\sigma^{\flat}$&$\zeta^{\flat} -\sigma^{\flat}$  \\
\hline
\end{tabular}}
\end{center}

\medskip

\begin{center}
{\renewcommand{\arraystretch}{1.5}
\renewcommand{\tabcolsep}{0.2cm}
\begin{tabular}{|c|c|c|c|c|}
\hline
 & $a_i$&$a'_i$&$b_i$&$b'_i$ \\
\hline
$\wt S$&$a_i +b'_{i-1}$&$a'_i +b_{i-1}$&$b'_{i-1}$& $b_i$  \\
\hline
$\wt T$&$a'_{i+1}$&$a_i$&$b_i +a'_{i+1}$& $b'_i +a_{i+1}$  \\
\hline
\end{tabular}}
\end{center}

\medskip

\begin{center}
{\renewcommand{\arraystretch}{1.5}
\renewcommand{\tabcolsep}{0.2cm}
\begin{tabular}{|c|c|c|c|}
\hline
 & $\tau_i$&$\breve{\sigma_i}$&$ \breve{\zeta_i}$ \\
\hline
$\wt S$&$-\tau_{i+1} $&$\breve{\sigma_i} + \breve{\zeta}_{i-1} $&$ \breve{\zeta}_{i+1}$  \\
\hline
$\wt T$&$ -\tau_{i-1}$&$\breve{\sigma}_{i-1}$& $\breve{\zeta}_i + \breve{\sigma}_{i+1}$  \\
\hline
\end{tabular}}
\end{center}

\medskip

When $q\geq 5$, one has

\medskip

\begin{center}
{\renewcommand{\arraystretch}{1.5}
\renewcommand{\tabcolsep}{0.2cm}
\begin{tabular}{|c|c|c|c|c|}
\hline
 & $p(\sigma_i)$ & $p(\sigma'_i)$ & $p(\zeta_i)$ & $p(\zeta'_i)$ \\
\hline
$\wt S^2$&$p(\sigma_i+\zeta_{i-1}+\zeta'_{i-1})$&$p(\sigma'_i+\zeta_{i-1} +\zeta'_{i+1})$&$p(\zeta_{i-1})$& $p(\zeta'_{i-1})$  \\
\hline
$\wt T^2$&$p(\sigma_{i+1})$&$p(\sigma'_{i+1})$&$p(\zeta_{i}+\sigma_{i+1}+\sigma'_{i+1})$& $p(\zeta'_i+\sigma_{i+1}+\sigma'_{i-1})$  \\
\hline
$\wt J$&$p(\zeta_{i-1})$&$p(\zeta'_{i+1})$&$-p(\sigma'_i)$&$-p(\sigma_i)$ \\
\hline
\end{tabular}}
\end{center}

\medskip

\begin{center}
{\renewcommand{\arraystretch}{1.5}
\renewcommand{\tabcolsep}{0.2cm}
\begin{tabular}{|c|c|c|c|c|}
\hline
 &$\sigma$&$\zeta$&$\sigma^{\flat}$&$\zeta^{\flat}$ \\
\hline
$\wt S^2$&$\sigma +2\zeta$&$\zeta$&$\sigma^{\flat}$&$\zeta^{\flat}$  \\
\hline
$\wt T^2$&$\sigma $&$\zeta +2\sigma$&$\sigma^{\flat}$&$\zeta^{\flat}$  \\
\hline
$\wt J$&$\zeta$&$-\sigma$&$\zeta^{\flat}$&$\sigma^{\flat}$\\
\hline
\end{tabular}}
\end{center}

\medskip

\begin{center}
{\renewcommand{\arraystretch}{1.5}
\renewcommand{\tabcolsep}{0.2cm}
\begin{tabular}{|c|c|c|c|c|}
\hline
 & $a_i$&$a'_i$&$b_i$&$b'_i$ \\
\hline
$\wt S^2$&$a_i +b_{i-1}+b'_{i-1}$&$a'_i +b_{i-1}+b'_{i+1}$&$b_{i-1}$& $b'_{i-1}$  \\
\hline
$\wt T^2$&$a_{i+1}$&$a'_{i+1}$&$b_i +a_{i+1}+a'_{i+1}$& $b'_i +a_{i+1}+a'_{i-1}$  \\
\hline
$\wt J$&$b_{i-1}$&$b'_{i+1}$&$-a'_i$&$-a_i$\\
\hline
\end{tabular}}
\end{center}

\medskip

\begin{center}
{\renewcommand{\arraystretch}{1.5}
\renewcommand{\tabcolsep}{0.2cm}
\begin{tabular}{|c|c|c|c|}
\hline
 & $\tau_i$&$\breve{\sigma}_i$&$ \breve{\zeta}_i$ \\
\hline
$\wt S^2$&$\tau_{i-1} $&$\breve{\sigma}_i +  \breve{\zeta}_i+\breve{\zeta}_{i-1} $&$ \breve{\zeta}_{i-1}$  \\
\hline
$\wt T^2$&$ \tau_{i+1}$&$\breve{\sigma}_{i+1}$& $\breve{\zeta}_i + \breve{\sigma}_i+\breve{\sigma}_{i+1}$  \\
\hline
$\wt J$&$-\tau_i$&$\breve{\zeta}_i$&$-\breve{\sigma}_i$ \\
\hline
\end{tabular}}
\end{center}

\medskip

\subsection{The subspaces $H^{st}_1$ and $H_{rel}$} 

The subspace $H^{st}_1$ of $H_1(M,\Sigma,\mathbb{Q})$ is spanned by $\sigma$ and $\zeta$. The action of ${\rm Aff} (M,\omega)$
on this subspace is given by the homomorphism from ${\rm Aff} (M,\omega)$ to $SL(M,\omega) \subset SL(2,\mathbb Z)$ and the standard
action of $SL(2,\mathbb Z)$ on $\mathbb Z \sigma \oplus \mathbb Z \zeta$.\\

The subspace $H_{rel}$ spanned by $\sigma^{\flat}$ and $\zeta^{\flat}$ supplements $H_1(M,\mathbb Q )$ in $H_1(M,\Sigma, \mathbb Q)$
and is invariant under the action of ${\rm Aff} (M,\omega)$. The action of ${\rm Aff} (M,\omega)$ on $H_{rel}$ is given by the permutation group of $\Sigma = \{A_{0,1}, A_{1,0},A_{1,1} \}$ when $q=3$, and by the permutation group of $\{A_{0,1}, A_{1,0}\}$ when
$q \geq 5$. In both cases, the kernel of the action is the subgroup $\mathbb Z /q \times \Gamma(2)$ of ${\rm Aff} (M,\omega)$.\\

\subsection{The subspace $H_{\tau}$}

Denote by $H_{\tau}$ the $(q-1)$-dimensional subspace of $H_1(M,\Sigma,\mathbb{Q})$ spanned by the $\tau_i$, $i\in \mathbb Z/q$ (recall
that $\sum_i \tau_i =0$). The formulas above show that $H_{\tau}$ is invariant under the action of ${\rm Aff} (M,\omega)$ and that 
 ${\rm Aff} (M,\omega)$ acts on this subspace through the finite cyclic group $\mathbb Z /2q$. The generator $1$ of  $\mathbb Z /2q$
acts by $\tau_i \mapsto -\tau_{i+\frac{q+1}{2}}$. The homomorphism from  ${\rm Aff} (M,\omega) \simeq \mathbb Z /q \times SL(M,\omega)$ to $\mathbb Z /2q$ is given by
$$ (g,{\rm Id}) \mapsto 2g,\quad (0,\wt T) \mapsto 1, \quad (0,\wt S) \mapsto -1,$$
when $q=3$ and
$$ (g,{\rm Id}) \mapsto 2g,\quad (0,\wt T^2) \mapsto 2, \quad (0,\wt S^2) \mapsto -2,\quad  (0,\wt J) \mapsto q, $$
when $q \geq 5$.

\subsection{The subspace $\breve{H}$}

We denote by $\breve H$ the $(2q-2)$-dimensional subspace of $H_1(M,\mathbb{Q})$ spanned by the $\breve{\sigma}_i, \breve{\zeta}_i$, $i \in \mathbb Z /q$. We have thus
$$ H_1(M,\Sigma,\mathbb{Q}) = H_{rel} \oplus H_1(M,\mathbb{Q}),$$
$$ H_1(M,\mathbb{Q}) = H^{st}_1 \oplus H_{\tau} \oplus \breve H \;.$$

In this subsection, we discuss  general considerations valid for all $q \geq 3$. The special case $q=3$ is treated in the next subsection.

\medskip

Let $\rho$ be a $q$-th root of unity different from $1$. Let $\breve{\sigma}(\rho)$, $\breve{\zeta}(\rho)$ the elements of 
$\breve{H} \otimes \mathbb C$ defined by
$$ \breve{\sigma}(\rho):= \sum_{\mathbb Z /q} \rho^{-i}\breve{\sigma}_i,\quad  \breve{\zeta}(\rho):= \sum_{\mathbb Z /q} \rho^{-i}\breve{\zeta}_i,$$
and let $\breve{H}(\rho)$ the subspace of $\breve{H} \otimes \mathbb C$ spanned by $\breve{\sigma}(\rho)$and  $\breve{\zeta}(\rho)$.
From the formulas above, we have, for $g \in \mathbb Z /q$
$$g.\bs(\rho) = \rho^g \bs(\rho), \quad g.\bz(\rho) = \rho^g \bz(\rho).$$

When $q=3$, we have
\begin{alignat*}{9}
&\wt S (\bs(\rho)) &\;=\; &\bs(\rho) + \rho^{-1} \bz(\rho), \quad \quad  \quad &\wt S (\bz(\rho)) &\;=\; \rho \; \bz (\rho), &\\
&\wt T (\bs(\rho)) &\;=\; &\rho ^{-1} \bs (\rho),  \quad &\wt T (\bz (\rho)) &\;=\; \bz (\rho) + \rho  \; \bs (\rho)&,
\end{alignat*}
while for $q\geq 3$, we have
\begin{alignat*}{9}
&\wt S^2 (\bs(\rho)) & \;=\; & \bs(\rho)  + (1+\rho^{-1}) \bz(\rho) , \quad \quad &\wt S^2 (\bz(\rho)) & \;=\;  \rho^{-1}  \bz (\rho), & \\
&\wt T^2 (\bs(\rho)) &\;=\; &\rho \; \bs (\rho),   \quad \quad &\wt T^2 (\bz (\rho)) &\;=\; \bz (\rho) + (1+\rho)  \; \bs (\rho),\\
&\wt J (\bs(\rho)) &\;=\;& \bz(\rho), \quad \quad &\wt J (\bz(\rho)) &\;=\; -\bs(\rho). &
\end{alignat*}

Each of the $(q-1)$ $2$-dimensional subspaces $\breve{H}(\rho)$ is thus invariant under the action of the affine group, and their direct sum is equal to $\breve{H} \otimes \mathbb C$. Over $\mathbb R$, $\breve{H} \otimes \mathbb R$ splits into the $\frac{q-1}{2}$
$4$-dimensional subspaces induced by the $\breve{H}(\rho) \oplus \breve{H}(\rho^{-1})$.

\begin{remark}
The determinant of the restriction to $\breve{H}(\rho)$ of $\wt S^2 \; \wt T^2$ is equal to $1$ and its trace is equal to $2(1+\rho + \rho^{-1})$. When $\rho = \exp \frac{2i \pi}{q}$ and $q \geq 5$, the trace is $>2$. In particular, the affine group does not act on $H_1^{(0)}$ through a finite group. Therefore, by~\cite{Moller}, the square-tiled surface is not totally degenerate for $q\geq 5$.
\end{remark}

\subsection{The subspace $\breve{H}$ for $q=3$}

We assume in this subsection that $q=3$. The formulas of Subsection 3.3 show that the set of $24$ vectors in $\breve{H}$:

$$R\;= \;\{\pm \bs_i,\;\pm \bz_i,\; \pm(\bs_i +\bz_{i-1}), \; \pm(\bs_i-\bz_{i+1}) \;; i \in \Zq \}$$
is invariant under the action of ${\rm Aff} (M,\omega)$. More precisely, we have

\medskip 

\begin{center}
{\renewcommand{\arraystretch}{1.5}
\renewcommand{\tabcolsep}{0.2cm}
\begin{tabular}{|c|c|c|c|c|}
\hline
 & $\bs_i$&$\bz_i$&$\bs_i + \bz_{i-1}$&$\bs_i-\bz_{i+1}$ \\
\hline
$g$&$\bs_{i+g}$&$\bz_{i+g}$&$\bs_{i+g} + \bz_{i-1+g}$&$\bs_{i+g}-\bz_{i+1+g}$ \\
\hline
$\wt S$&$\bs_i + \bz_{i-1}$&$\bz_{i+1}$&$\bs_i-\bz_{i+1}$& $\bs_i$  \\
\hline
$\wt T$&$\bs_{i-1}$&$\bz_i + \bs_{i+1}$&$\bz_{i+2} - \bs_{i+1}$& $-\bz_{i+1}$  \\
\hline
\end{tabular}}
\end{center}

\medskip

The set $R$ is a root system of $D_4$ type. More precisely, let $\mathbb V$ be the set of nonzero vectors in $(\mathbb Z /3)^2$. Define a map $\varepsilon: \mathbb V \rightarrow \breve H$ by
\begin{eqnarray*}
\varepsilon(1,0) &=& \frac12 (-\bs_2 +\bz_0 -\bz_1),\\
\varepsilon(1,1) &=& \frac12 (-\bs_2 -\bz_2),\\
\varepsilon(1,2) &=& \frac12 (-\bs_2 + \bz_2),\\
\varepsilon(0,1) &=& \frac12 (-\bs_0 + \bs_1 -\bz_2),
\end{eqnarray*}
and $\varepsilon(-v) = -\varepsilon(v)$ for $v \in \mathbb V$.

\medskip

One has then 
\begin{equation*}
R= \{ \varepsilon(v) + \varepsilon(v')\,; \; v,v' \in \mathbb V ,v \ne \pm v'\;\}.
\end{equation*}

The action of the affine group on $\breve H$ is therefore given by a homomorphism $Z$ from ${\rm Aff} (M,\omega)$ to the automorphism group $A(R)$. This last group and the Weyl group $W(R) \subset A(R)$ have been described in Subsection 2.7. We will now describe $Z$.\\

Denote by $\overline S$, $\overline T$ the images of $S,T$ in $SL(2, \mathbb Z /3)$. This group acts on $\mathbb V$. One checks that,
for all $v \in \mathbb V$, one has

$$ \wt S ( \varepsilon(v)) = \varepsilon(\overline S (v)),\quad \quad \wt T ( \varepsilon(v)) = \varepsilon(\overline T (v)).$$

On the other hand, one has

$$Z(1).\varepsilon(1,0)=1.\varepsilon(1,0) = \frac12 (-\varepsilon(1,0) + \varepsilon(1,1) -\varepsilon(1,2) + \varepsilon (0,1)).$$

We obtain thus

\begin{lemma}
\begin{enumerate}
\item The inverse image by $Z$ of $W(R)$ in the affine group is the subgroup $\simeq SL(2, \mathbb Z)$ fixing the $A_i$.
\item The kernel of $Z$ is the congruence subgroup
$$ \Gamma(3):=\left\{M\equiv \textrm{Id}_{SL(2,\mathbb{Z})} \, \textrm{mod}\, 3\right\}.$$
\item The image of $Z$ is the product of $\mathbb Z /3 \subset A(R)/W(R) \simeq S_3$ by $SL(2,\mathbb Z /3)$, this last group
acting in the natural way on $\{\varepsilon (v); v\in \mathbb V\;\}$.
\end{enumerate}
\end{lemma}

\begin{proof}
The formulas above show that the images by $Z$ of $\wt S,\, \wt T$ belong to $W(R)$, but that the image by $Z$ of the generator
$1$ of $\mathbb Z /3$ does not belong to $W(R)$. This proves the first assertion of the lemma. The other assertions of the lemma then follow from the formulas for $\wt S,\, \wt T$.
\end{proof} 

\begin{lemma}The intersection of the image of $Z$ with the Weyl group $W(R)$ is \emph{exactly} the subgroup of symplectic elements of $W(R)$ with respect to the restriction of the intersection form to $\breve H$.
\end{lemma}

\begin{proof}We begin by computing  the intersection form on the vectors $\varepsilon(v)$. We write $\breve{\sigma}_i = \gamma_i - \gamma_{i+1}$ and $\breve{\zeta_i}= \delta_{i-1}-\delta_i$ where $\gamma_i:=\sigma_i+\sigma_{i-1}'$ and $\delta_i:=\zeta_i+\zeta_{i+1}'$. A direct calculation (with Figure~\ref{coordinates}) shows that, for $i\in \mathbb Z /3$:  

$$(\gamma_i,\gamma_{i+1}) = (\delta_i,\delta_{i+1})=2\;,$$
$$(\gamma_{i},\delta_{i})=(\gamma_{i},\delta_{i+1})=-(\gamma_{i},\delta_{i-1})=1\;.$$

It follows that, for $i\in \mathbb Z /3$:

$$(\breve{\sigma}_i,\breve{\sigma}_{i+1})=(\breve{\zeta}_i,\breve{\zeta}_{i+1})=6\;,$$
$$(\breve{\sigma}_{i},\breve{\zeta}_{i})=-4\,,\quad (\breve{\sigma}_{i},\breve{\zeta}_{i+1})=(\breve{\sigma}_{i},\breve{\zeta}_{i-1})=2\;.$$

From this, we see that, for $v,v' \in \mathbb V$, $v \ne \pm v'$:

$$ \begin{array}{ll}
&  (\varepsilon(v),\varepsilon(v'))=\left\{ \begin{array}{l} 2 \quad \quad \; \;\hbox{if} \; \det (v,v')=1 \;\;\hbox{mod}\;3,\\[2ex]
-2 \quad \quad \hbox{if} \; \det (v,v')=-1 \;\hbox{mod}\;3.
 \end{array} \right. 
\end{array}$$

As expected, the action of $SL(2,\mathbb Z /3)$ preserves the symplectic form. We want to show that no other element of $W(R)$ does so. We have the exact sequences (see Subsection 2.7):
$$1 \longrightarrow (\mathbb Z /2)^3 \longrightarrow W(R) \longrightarrow S_4 \longrightarrow 1\;,$$
$$1 \longrightarrow \mathbb Z /2 \longrightarrow SL(2,\mathbb Z /3) \longrightarrow A_4 \longrightarrow 1 \;.$$

Thus, it is sufficient to see that, if $g \in W(R)$ preserves the symplectic form and $g(\varepsilon(1,0))= \varepsilon(1,0)$, $g(\varepsilon(0,1))= \eta\;\varepsilon(0,1)$, $\eta = \pm 1$, then $g= 1$. Indeed, as $g$ preserves the symplectic form, we must have $\eta = 1$, $g(\varepsilon(1,1))= \varepsilon(1,1)$, $g(\varepsilon(1,-1))= \varepsilon(1,-1)$.
\end{proof}

The proof of Theorem ~\ref{t.B} is now complete.

\appendix\section{Parity of the spin structure}\label{a.g4parity}
 After the work of Kontsevich and Zorich~\cite{KZ} on the complete classification of the connected components of the strata of the moduli space of Abelian differentials, we know that the stratum $H(1,1,1,1)$ of $(M_3,\omega_{(3)})$ is connected, but the stratum $H(2,2,2)$ of $(M_4,\omega_{(4)})$ has exactly two connected components distinguished by a topological invariant called the \emph{parity of the spin structure}. After  recalling the definition of the parity of the spin structure , we will show that the parity of spin of $(M_4,\omega_{(4)})$ is \emph{even}.

Let $(M,\omega)$ be a translation surface such that the zeros $p_1,\dots,p_n$ of $\omega$ have even orders $2l_1,\dots,2l_n$. For  an oriented smoothly immersed closed curve $\gamma$ avoiding the zeros of $\omega$, the index $\textrm{ind}_{\omega}(\gamma)\in\mathbb{Z}$ is the integer such
that the total change of angle between the tangent vector of $\gamma$ and the rightwards horizontal direction determined by $\omega$ is $2\pi\cdot\textrm{ind}_{\omega}(\gamma)$. Because the zeros of $\omega$ have even order, the parity of $\textrm{ind}_{\omega}(\gamma)$ depends only on the class $[\gamma]$ of $\gamma$ in $H_1(M,\mathbb Z)$ and will be denoted by $\textrm{ind}_{\omega}([\gamma]) \in \mathbb Z /2$.

In this context, fix $(\alpha_i, \beta_i)_{i=1,\dots,g}$ a symplectic basis of the integral homology group $H_1(M,\mathbb{Z})$. Let
\begin{equation*}
\phi(\omega):=\sum\limits_{i=1}^g (\textrm{ind}(\alpha_i)+1)\cdot(\textrm{ind}(\beta_i)+1) \quad (\textrm{mod}\, 2).
\end{equation*}
The quantity $\phi(\omega)\in\mathbb{Z}/2$ is called the \emph{parity of the spin structure} of $\omega$. It can be shown that $\phi(\omega)$ doesn't depend on the choice of the symplectic basis $(\alpha_i, \beta_i)_{i=1,\dots,g}$ (in fact it depends only on the spin structure $\kappa(\omega)=l_1\cdot p_1+\dots +l_n\cdot p_n\in \textrm{Pic}(M)$ of $\omega$). For further details and alternative definitions of the parity of the spin structure (giving the motivation for the name ``parity of the spin structure'' of $\phi(\omega)$), see~\cite{KZ}.

We now compute the parity of the spin structure of $(M_4,\omega_{(4)})$. Consider the following picture describing the oriented paths $\sigma_i$, $\zeta_i$, $\sigma_i'$, $\zeta_i'$ ($i \in \mathbb Z /3$) and their relative positions, when one makes a counter-clockwise $6\pi$ turn around the singularity 
$A_{1,1}$ (see Figure~\ref{FM-separatrix} below).
\begin{figure}[!h]
\centering
\includegraphics[scale=0.3]{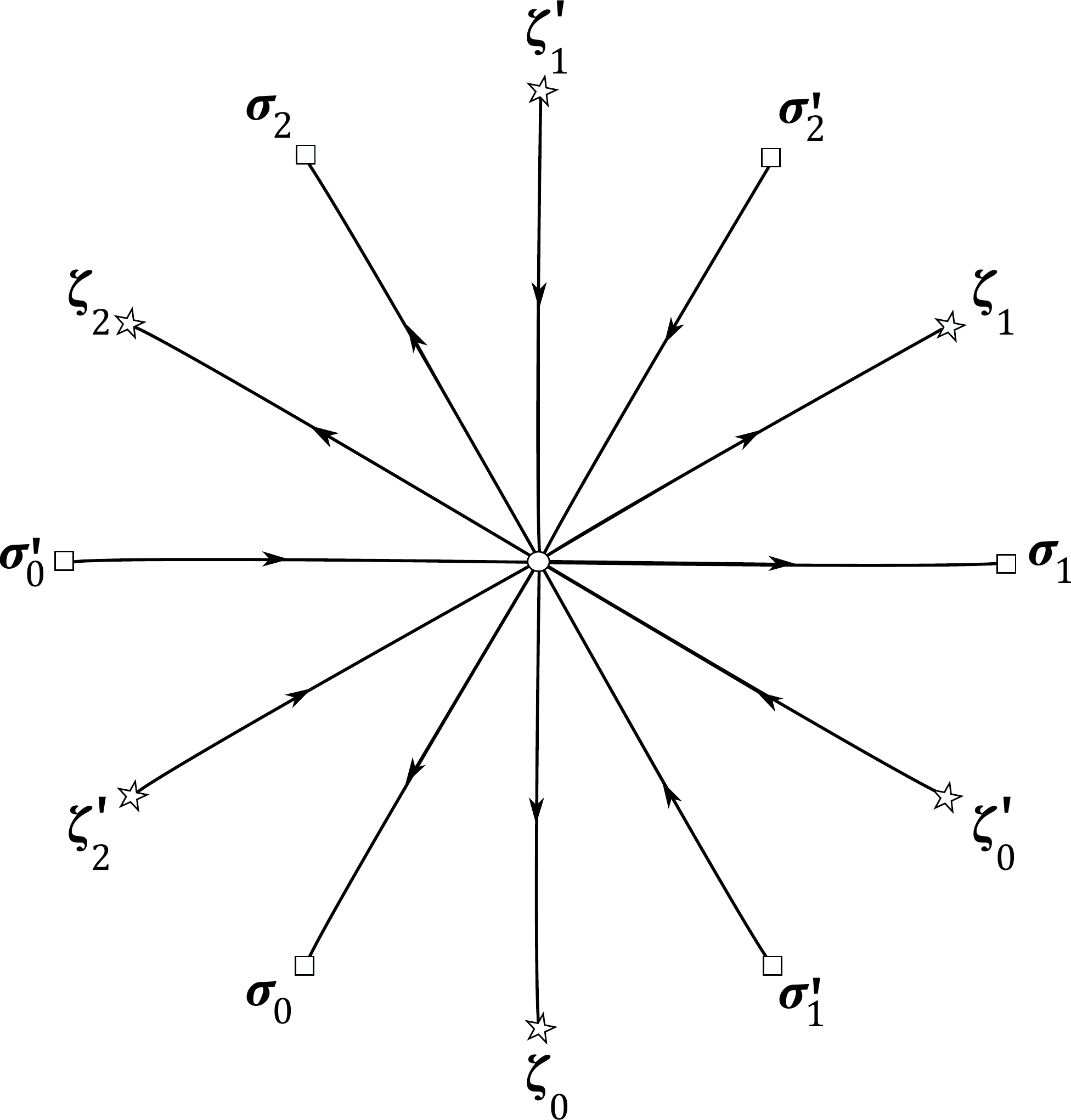}
\caption{Ingoing and outgoing separatrix rays nearby the singularity $A_{1,1}$.}\label{FM-separatrix}
\end{figure}

Here, the singularity $A_{1,1}$ is located at the center of the figure; it  is linked to $A_{0,1}$ by the $\sigma_i$, $\sigma'_i$ and to $A_{1,0}$
by the $\zeta_i$, $\zeta'_i$.\\ 

A similar picture can be depicted nearby the other two singularities $A_{1,0}$ and $A_{0,1}$.

\begin{figure}[!h]
\centering
\includegraphics[scale=0.32]{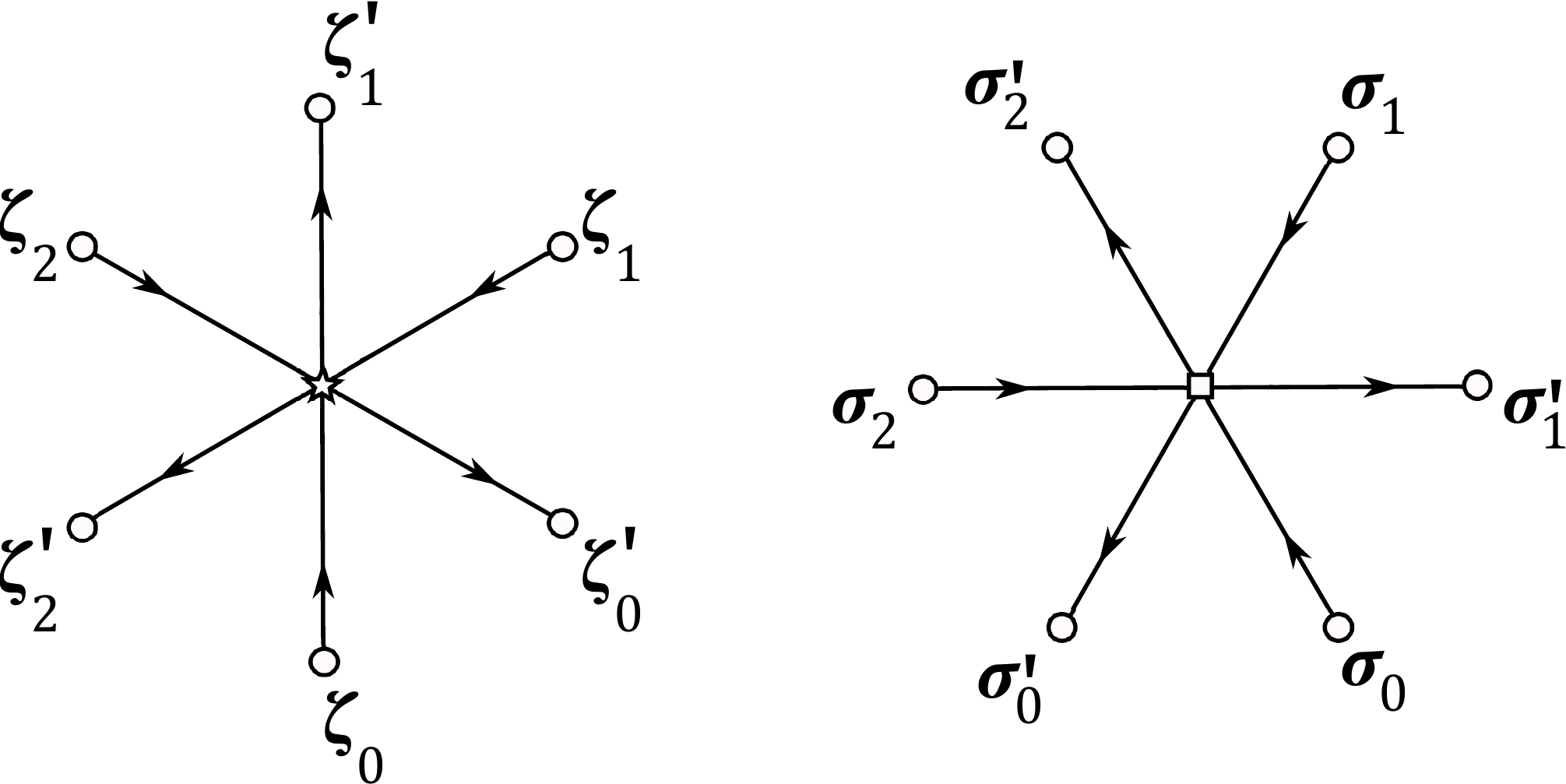}
\caption{Separatrix diagrams nearby the singularities $A_{1,0}$ and $A_{0,1}$.}\label{FM-separatrices}
\end{figure}

Here, $A_{1,0}$ and $A_{0,1}$ (resp.) are located at the center of the left-hand side and right-side of the picture (resp.).\\

Let 
\begin{align*} 
&\alpha_1:=p(\sigma_1+\sigma_1'),\;\;  &\alpha_2&:=p(\sigma_2+\sigma_2'),\;\;   &\alpha_3&:=p(\sigma_0+\sigma_0'),\;\; &\alpha&:=p(\sigma_2+\sigma_0') ,\\ 
&\beta_1:=p(\zeta_0+\zeta_0'),\;\;  &\beta_2&:=p(\zeta_1+\zeta_1'),\;\; &\beta_3&:=p(\zeta_2+\zeta_2'),\;\; &\beta&:=p(\zeta_1+\zeta_0').
\end{align*}

It is easy to check that $(\alpha_1,\alpha_2,\alpha_3,\alpha,\beta_1,\beta_2,\beta_3,\beta)$ is an integral basis of $H_1(M,\mathbb Z)$. From 
Figure~\ref{FM-separatrix}, the intersection pairing is given by

\medskip

\begin{center}
{\renewcommand{\arraystretch}{1.5}
\renewcommand{\tabcolsep}{0.2cm}
\begin{tabular}{|c|c|c|c|c|c|c|c|c|}
\hline
$(.,.)$ & $\alpha_1$ & $\alpha_2$ & $\alpha_3$ & $\alpha$ & $\beta_1$  & $\beta_2$ & $\beta_3$ & $\beta$\\
\hline
$\alpha_1$ & $0$ & $0$ & $0$ & $0$ & $1$ & $0$ & $0$ & $1$\\
\hline
$\alpha_2$ & $0$ & $0$ & $0$ & $-1$ & $0$ & $1$ & $0$ & $0$\\
\hline
$\alpha_3$ & $0$ & $0$ & $0$ & $1$ & $0$ & $0$ & $1$ & $0$ \\
\hline
$\alpha$ & $0$ & $1$ & $-1$ & $0$ & $0$ & $0$ & $1$ & $0$ \\
\hline
$\beta_1$ & $-1$ & $0$ & $0$ & $0$ & $0$ & $0$ & $0$ & $-1$ \\
\hline
$\beta_2$ & $0$ & $-1$ & $0$ & $0$ & $0$ & $0$ & $0$ & $1$\\
\hline
$\beta_3$ & $0$ & $0$ & $-1$ & $-1$ & $0$ & $0$ & $0$ & $0$\\
\hline
$\beta$ & $-1$ & $0$ & $0$ & $0$ & $1$ & $-1$ & $0$ & $0$\\
\hline
\end{tabular}}
\end{center}

\medskip

We define
$$\alpha_4:=\alpha -\alpha_3 +\beta_2 -\beta_3, \quad \beta_4:=\beta-\alpha_1+\alpha_2-\beta_1\;.$$
Then $(\alpha_i,\beta_i)_{i=1,\dots,4}$ is a symplectic basis of $H_1(M_4,\mathbb{Z})$.

It is clear that $\textrm{ind}_{\omega}(\alpha_i)=\textrm{ind}_{\omega}(\beta_i)=0$ for $i=1,2,3$. 
Furthermore, if we follow a path close to $-\sigma_0$, $-\zeta'_2$, $-\zeta_2$, $\zeta_1$, $\zeta'_1$ and $\sigma_2$ (in this order) 
avoiding the zeroes of $\omega_{(4)}$, we see that $\textrm{ind}_{\omega}(\alpha_4)=0$. If we follow a similar path close to $\zeta_1$, $-\zeta_0$, 
$-\sigma'_1$, $-\sigma_1$, $\sigma_2$, $\sigma'_2$, (in this order), we see that $\textrm{ind}_{\omega}(\beta_4)=0$.

We conclude that the parity of the spin of $(M_4,\omega_{(4)})$ is even.
\begin{remark}
Another way to compute the parity of the spin structure could have been to determine the Rauzy class of an interval exchange map obtained as the first return map of a nearly vertical linear flow of $(M,\omega)$ on an appropriate transversal. See Appendix C of~\cite{Zorich2} for a detailed explanation of this method (including basic definitions and main technical steps). For sake of convenience of the reader, we compute below some combinatorial data associated to $(M_4,\omega_{(4)})$ needed to perform the arguments of~\cite{Zorich2}.

Fix a small angle $\theta$; we consider, as shown on Figure~\ref{Rauzy-g4} below, the first return of the upwards vertical flow for $e^{i\theta}\omega_{(4)}$ on a transverse interval $L$ which is horizontal for $e^{i\theta}\omega_{(4)}$.

\begin{figure}[!h]
\centering
\includegraphics[scale=0.35]{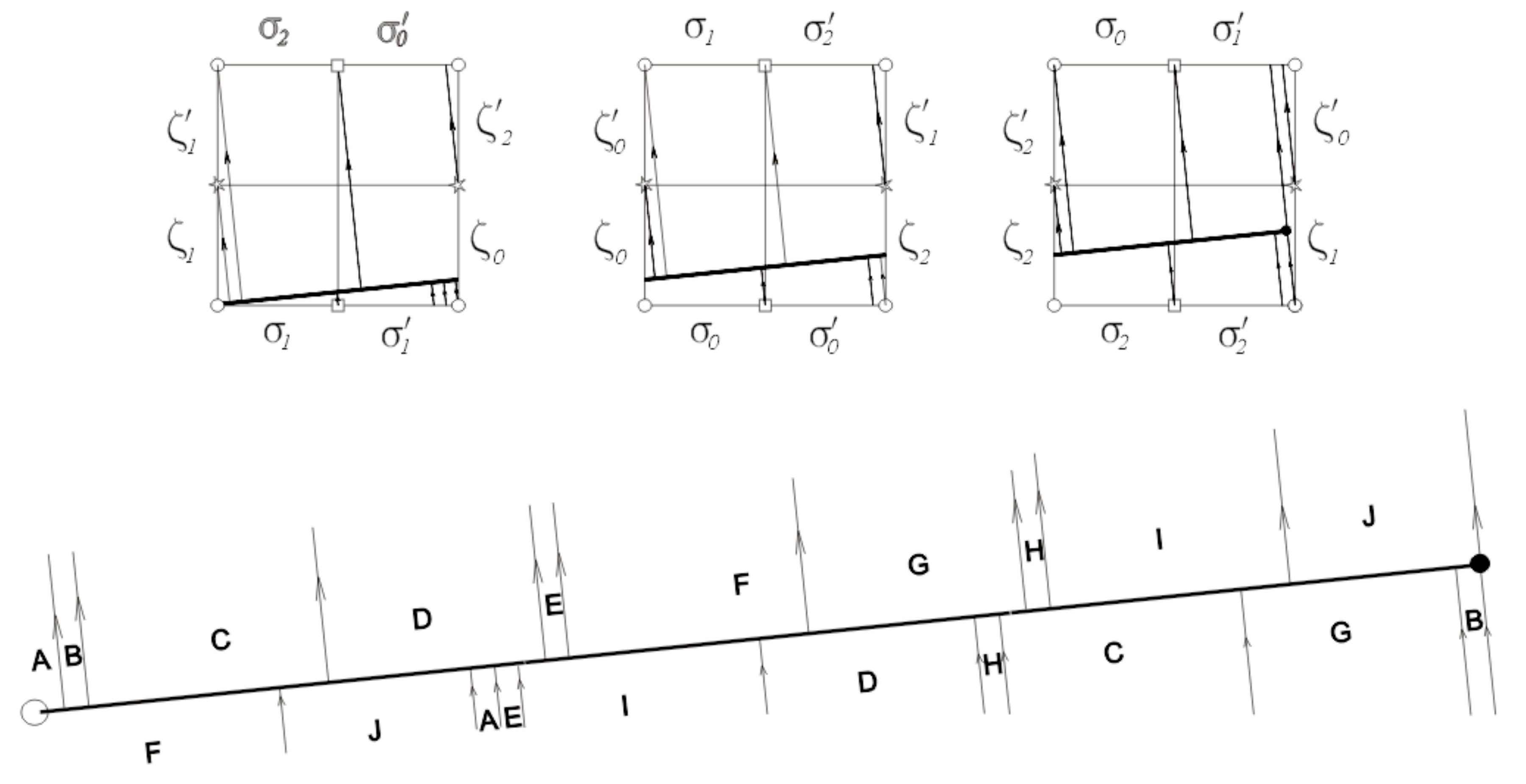}
\caption{The transverse interval $L$ and the associated return map }\label{Rauzy-g4}
\end{figure}

Here, the interval $L$ starts at the singularity $\bigcirc$ and ends at the black diamond dot. Also, we indicated the "top" partition of $L$ (corresponding to the domain of the return map) into $10$ intervals (named $A$, $B$, $C$, $D$, $E$, $F$, $G$, $H$, $I$ and $J$) and their images by the return map (the "bottom" partition of $L$). The combinatorial data of the return map is thus
\begin{displaymath}
\left( \begin{array}{cccccccccc}
A & B & C & D & E & F & G & H & I & J \\
F & J & A & E & I & D & H & C & G & B \\
\end{array} \right)
\end{displaymath}
\end{remark}

\section{$\textrm{Aff}(M,\omega)$-invariant supplements to $H_1(M,\mathbb{R})$}\label{a.supplement}
We have seen that, for both $(M_3,\Sigma_{(3)},\omega_{(3)})$ and $(M_4,\Sigma_{(4)},\omega_{(4)})$, the absolute homology group $H_1(M,\mathbb{R})$ admits a $\textrm{Aff}(M_k,\omega_{(k)})$-invariant supplement $H_{rel}$ within the relative homology group $H_1(M,\Sigma,\mathbb{R})$. This appendix  shows a simple example of a square-tiled surface of genus $2$ with two simple zeros where this does not hold.\\

We consider the square-tiled surface $(M,\omega)$ obtained  from the polygon $P$ with vertices $A_0=(0,0),A^t_1= (1,2),A^t_2=(2,3),A^t_3=(3,3),A^t_4=(4,2),A_5=(5,1),A^b_4=(4,-1),A^b_3=(3,-2),A^b_2=(2,-2), A^b_1=(1,-1)$ by identifying parallel opposite sides (this sides are also symmetric w.r.t. the center of symmetry $(\frac 52 , \frac 12 )$ of the polygon, but the identifications are through translations). The two simple zeros on the surface are $A_{even} = A_0=A^t_2=A^t_4=A^b_4=A^b_2$ and 
$A_{odd} = A^t_1=A^t_3=A_5=A^b_3=A^b_1$. We have thus $\Sigma = \{A_{even}, A_{odd} \}$. The reader can find a pictorial description of $(M,\omega)$ in Figure~\ref{supplement-H(1,1)} below:
\begin{figure}[!h]
\centering
\includegraphics[scale=0.35]{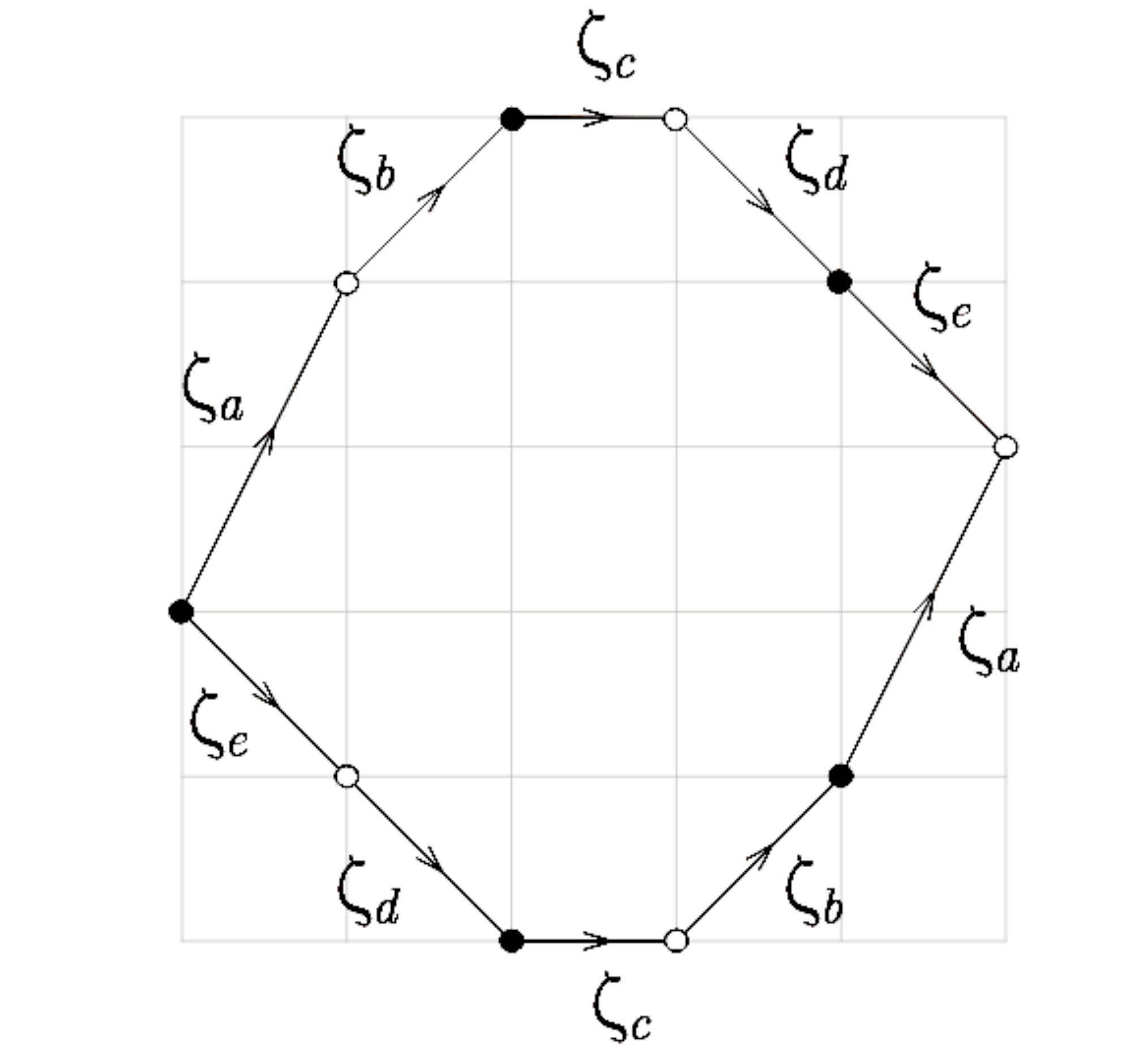}
\caption{Square-tiled surface of $\mathcal{M}_{(1,1)}$.}\label{supplement-H(1,1)}
\end{figure}
\begin{proposition}There is no $\textrm{Aff}(M,\omega)$-invariant supplement of the absolute homology $H_1(M,\mathbb{R})$ inside the relative homology $H_1(M,\Sigma,\mathbb{R})$.
\end{proposition}
\begin{proof} Let $\textrm{Aff}_+(M,\omega)$ be the subgroup of $\textrm{Aff}(M,\omega)$ of index $\leq 2$ formed by the elements which fix each of the two zeros $A_{odd}$ and $A_{even}$. It acts trivially on $H_0(\Sigma,\mathbb{R})$, and therefore also on any $\textrm{Aff}(M,\omega)$-invariant supplement of $H_1(M,\mathbb{R})$ inside $H_1(M,\Sigma,\mathbb{R})$.
Denote by $\zeta_a,\zeta_b,\zeta_c,\zeta_d,\zeta_e$ the relative homology classes represented respectively by the orienred sides $[A_0 A^t_1]$ 
(or $[A^b_4 A_5]$),$[A^t_1 A^t_2]$ (or $[A^b_3 A^b_4]$), $[A^t_2 A^t_3]$ (or $[A^b_2 A^b_3]$), $[A^t_3 A^t_4]$ (or $[A^b_1 A^b_2]$), $[A^t_4 A_5]$ 
(or $[A_0 A^b_1]$) of $P$. They form a basis of $H_1(M,\Sigma,\mathbb{Z})$. We have $\partial \zeta_a =-\partial \zeta_b=\partial \zeta_c=-\partial \zeta_d=\partial \zeta_e=A_{odd}-A_{even}$.
 
\medskip

The classes $\zeta_0:=2(\zeta_a -\zeta_e) -3(\zeta_b -\zeta_d)$, 
$\zeta_1:= \zeta_a -3 \zeta_c +2 \zeta_e$ form a basis of $H_1^{(0)}(M,\mathbb Q)$. Denote also by $\zeta^*$ the class $\zeta_e - \zeta_d$ with boundary
$2(A_{odd}-A_{even})$, satisfying also $\int _{\zeta^*} \omega =0$. We have to show that there is no absolute class $\zeta \in  H_1(M,\mathbb{R})$
such that $\zeta^* + \zeta$ is invariant under $\textrm{Aff}_+(M,\omega)$. Assume that such a $\zeta$ exists. Obviously, we must have $\int _{\zeta^* +\zeta} \omega =0$, hence $\zeta \in H_1^{(0)}(M,\mathbb R)$. We write $\zeta = s_0 \zeta_0 + s_1 \zeta_1$, for some $s_0,s_1 \in \mathbb R$.

\medskip

To derive a contradiction, we will consider parabolic elements of the Veech group associated to three distinct cusps of the Teichm\"uller curve associated to $(M,\omega)$. 

\begin{itemize}
\item {\bf vertical decomposition}: the surface splits into three cylinders of height $1$,  respective widths $3,8,5$, whose associated homology classes are respectively 
$\zeta_a - \zeta_e$, 
$2(\zeta_a - \zeta_e) + \zeta_b - \zeta_d$, $\zeta_a - \zeta_e + \zeta_b - \zeta_d$. There is therefore an element $A_{vert}$ of $\textrm{Aff}_+(M,\omega)$ with linear part $\left( \begin{array}{cc} 1 & 0 \\  120 & 1 \\  \end{array} \right)$.

One checks that
$$(A_{vert} - {\rm Id})(\zeta^*)= 5 \zeta_0,\quad (A_{vert} - {\rm Id})(\zeta_0)= 0,\quad (A_{vert} - {\rm Id})(\zeta_1)= 24 \zeta_0.$$

\item {\bf horizontal decomposition}: the surface splits into two cylinders of height $1$, respective widths $4,12$, whose associated homology classes are respectively 
$\zeta_b + 2 \zeta_c + \zeta_d$, 
$3(\zeta_b + \zeta_c + \zeta_d) + \zeta_a + 2\zeta_e$. There is therefore an element $A_{hor}$ of $\textrm{Aff}_+(M,\omega)$ with linear part $\left( \begin{array}{cc} 1 & 12 \\  0 & 1 \\  \end{array} \right)$.

One checks that
$$(A_{hor} - {\rm Id})(\zeta^*)= - \zeta_1,\quad (A_{hor} - {\rm Id})(\zeta_0)= 6\zeta_1,\quad (A_{hor} - {\rm Id})(\zeta_1)=0.$$

\item {\bf main diagonal (slope $1$) decomposition}: the surface splits into two cylinders of respective (normalized) heights $1,2$, respective (normalized) widths $4,6$, whose associated homology classes are respectively 
$\zeta_a + 2 \zeta_b + \zeta_c$, 
$2(\zeta_a + \zeta_b + \zeta_c) + \zeta_d -\zeta_e$. There is therefore an element $A_{diag}$ of $\textrm{Aff}_+(M,\omega)$ with linear part $\left( \begin{array}{cc} -11 & 12 \\  -12 & 13 \\  \end{array} \right)$.

One checks that
$$(A_{diag} - {\rm Id})(\zeta^*)= 0,\; (A_{diag} - {\rm Id})(\zeta_0)= \frac 23 (\zeta_1 - 2 \zeta_0), \; (A_{diag} - {\rm Id})(\zeta_1)=\frac 43 (\zeta_1 - 2 \zeta_0).$$
\end{itemize}

If $\zeta^* + s_0 \zeta_0 +s_1 \zeta_1$ was killed by $A_{vert} - {\rm Id}$, $ A_{hor} - {\rm Id}$ and $ A_{diag} - {\rm Id}$, we would have 
$s_1 = - \frac{5}{24}$, $s_0 = \frac 16$, $s_0 +2s_1 =0$, a contradiction which concludes the proof of the proposition.
\end{proof}


\begin{thebibliography}{99}


\bibitem{AV} A. ~Avila and M.~Viana,
\emph{Simplicity of {L}yapunov Spectra: Proof of the {Z}orich-{K}ontsevich conjecture},
Acta Math., v. 198 (2007), 1--56.


\bibitem{Bauer} O.~Bauer, \emph{Familien von Jacobivariet\"aten \"uber Origamikurven}, PhD thesis available at http://digbib.ubka.uni-karlsruhe.de/volltexte/1000011870, 2009.

\bibitem{BB} J.~Borwein and P. Borwein, \emph{Pi and the AGM}, Canadian Math. Soc. Series of Monographs and Advanced Texts, John Willey and Sons, New York, 1987.

\bibitem{Bourbaki} N.~Bourbaki, \emph{Groupes et alg\`ebres de Lie. Chapitre VI: syst\`emes de racines.}, Hermann, Paris, 1960.

\bibitem{BM} I.~Bouw and M. M\"oller, \emph{Teichm\"uller curves, triangle groups, and Lyapunov exponents}, to appear in Annals of Math.




\bibitem{Forni} G.~Forni, \emph{Deviation of ergodic averages for
area-preserving flows on surfaces of higher genus}, Ann. of Math.,
v.155 (2002), no. 1, 1--103.

\bibitem{ForniSurvey} \bysame, \emph{On the Lyapunov exponents of the Kontsevich-Zorich cocycle},
Handbook of Dynamical Systems v. 1B, B.~Hasselblatt and A.~Katok, eds., Elsevier (2006), 549--580.

\bibitem{FM} G.~Forni and C.~Matheus, \emph{An example of a Teichmuller disk in genus 4 with degenerate Kontsevich-Zorich spectrum}, preprint 2008, http://arxiv.org/abs/0810.0023.

\bibitem{HS} F.~Herrlich and G.~Schmith\"usen, \emph{An extraordinary origami curve}, Math. Nachr., v. 281 (2008) no. 2, 219--237.

\bibitem{HuSc} P.~Hubert and T.~Schmidt, \emph{An introduction to Veech surfaces}, Handbook of Dynamical Systems v. 1B, B.~Hasselblatt and A.~Katok, eds., Elsevier (2006), 501--526.
        

\bibitem{K} M.~Kontsevich, \emph{Lyapunov exponents and {H}odge theory}, in
`The mathematical beauty of physics', Saclay, 1996. Adv. Ser. Math. Phys. v. 24,318--332,
World Scientific, River Edge, NJ, 1997.

\bibitem{KZ} M.~Kontsevich and A.~Zorich, \emph{Connected components of
the moduli spaces of Abelian differentials with prescribed
singularities}, Invent. Math.  v. 153 (2003) no. 3, 631--678.

\bibitem{Lanneau} E.~Lanneau, \emph{Connected components of the strata of the moduli spaces of quadratic differentials}, Ann. Sci. ENS, v. 41 (2008), 1-56.

\bibitem{Masur1} H.~Masur, \emph{Interval exchange transformations and measured foliations}, Ann. of
 Math., v. 115 (1982), 169--200.


\bibitem{Moller} M.~M\"oller, \emph{Shimura and Teichm\"uller curves}, preprint, 2005, http://arxiv.org/math.AG/0501333

\bibitem{Veech} W.~Veech, \emph{Teichm\"uller Geodesic Flow}, Ann.
of Math., v. 124 (1986), 441--530.

\bibitem{Veech1} \bysame, \emph{Gauss measures for transformations on the space of interval exchange maps}, Ann. of Math., v. 115, 201--242, 1982.

\bibitem{Veech2} \bysame, \emph{Moduli spaces of quadratic differentials}, J. Anal. Math., v. 55 (1990), 117--171.

\bibitem{Veech3} \bysame, \emph{Teichm\"uller curves in moduli space, Eisenstein series and an application to triangular billiards}, Inv. Math., v. 97 (1989), 553--583.

\bibitem{Y-Pisa} J. C. Yoccoz, \emph{Interval exchange maps and translation surfaces}, Clay Math. Inst. Summer School on Homogenous Flows, Moduli Spaces and Arithmetic, Pisa (2007), available at http://www.college-de-france.fr/media/equ$\_$dif/UPL15305$\_$PisaLecturesJCY2007.pdf

\bibitem{Zorich1} A.~Zorich, \emph{Asymptotic flag of an orientable measured foliation on a surface},
in `Geometric Study of Foliations', World Scientific (1994), 479--498.

\bibitem{Zorich2} \bysame, \emph{Explicit Jenkins-Strebel representatives of all strata of Abelian and quadratic differentials},
Journal of Modern Dynamics, v. 2 (2008), 139--185.

\bibitem{Zorich3} \bysame, \emph{Flat Surfaces},
Frontiers in number theory, physics, and geometry, Springer (2006), 437--583.






\end{thebibliography}
\end{document}